\documentclass[twoside,11pt]{article}

\usepackage[preprint]{jmlr2e}

\usepackage{jmlr2e}
\usepackage{hyperref}       
\hypersetup{colorlinks=false, linkcolor=blue, breaklinks=true, urlcolor=blue}
\usepackage{amsmath}
\usepackage{enumerate}
\usepackage{doi,xspace}
\usepackage{stmaryrd}
\usepackage{multirow}
\usepackage{pgfplots}
\usepackage{pgf}
\usepackage{pgfmath}
\usepackage{tikz}

\newcommand{\KM}{Krasnosel’skii–Mann\xspace}

\newcommand{\suchthat}{\;\ifnum\currentgrouptype=16 \middle\fi|\;}
\newcommand{\until}[1]{\{1,\dots, #1\}}
\newcommand{\subscr}[2]{#1_{\textup{#2}}}

\newcommand{\setdef}[2]{\{#1 \; | \; #2\}}

\newcommand{\map}[3]{#1: #2 \rightarrow #3}

\newcommand{\real}{\mathbb{R}}
\newcommand{\realpositive}{\mathbb{R}_{>0}}
\newcommand{\realnonnegative}{\mathbb{R}_{\geq0}}
\newcommand{\complex}{\mathbb{C}}

\newcommand{\scirc}{\raise1pt\hbox{$\,\scriptstyle\circ\,$}}

\newcommand{\prox}{\mathrm{prox}}

\newcommand{\Iinfty}{I_{\infty}}

\newcommand{\sign}{\operatorname{sign}}

\DeclareSymbolFont{bbold}{U}{bbold}{m}{n}
\DeclareSymbolFontAlphabet{\mathbbold}{bbold}
\newcommand{\vect}[1]{\mathbbold{#1}}

\newcommand{\vectorzeros}[1][]{\vect{0}_{#1}}

\newcommand{\ds}{\displaystyle}

\newtheorem{algo}[theorem]{Algorithm}

\makeatletter
\def\@opargbegintheorem#1#2#3{\trivlist
	\item[]{\bfseries #1\ #2\ (#3)} \itshape}
\makeatother

\graphicspath{{./images/}{./fig}}

\newcommand{\jac}[1]{D\mkern-0.75mu{#1}}

\newcommand{\WP}[2]{\left\llbracket{#1}, {#2}\right\rrbracket}

\newcommand{\seminorm}[1]{{\left\vert\kern-0.25ex\left\vert\kern-0.25ex\left\vert #1
		\right\vert\kern-0.25ex\right\vert\kern-0.25ex\right\vert}}

\newcommand{\semimeasure}[1]{\mu_{\seminorm{\cdot}}\kern-0.5ex\left(#1\right)}

\newcommand{\inprod}[2]{\langle\!\langle{#1},{#2}\rangle\!\rangle}

\newcommand{\osL}{\operatorname{osL}}
\newcommand{\Lip}{\operatorname{Lip}}

\newcommand{\spectrum}{\operatorname{spec}}

\newcommand{\norm}[2]{\|#1\|_{#2}}

\newcommand{\zero}{\operatorname{Zero}}
\newcommand{\fixed}{\operatorname{Fix}}
\DeclareMathOperator*{\argmin}{arg\,min}

\DeclareMathOperator{\dom}{Dom}

\newcommand{\diagL}{\operatorname{diagL}}

\newcommand{\Id}{\mathsf{Id}}

\newcommand{\OF}{\mathsf{F}}
\newcommand{\OG}{\mathsf{G}}

\newcommand{\ON}{\mathsf{N}}
\newcommand{\OT}{\mathsf{T}}

\newcommand{\OR}{\mathsf{R}}
\newcommand{\OJ}{\mathsf{J}}
\newcommand{\OS}{\mathsf{S}}

\newcommand{\bigO}{\mathcal{O}}

\newcommand{\mcX}{\mathcal{X}}

\newcommand{\mcY}{\mathcal{Y}}
\newcommand{\change}[1]{#1}

\usepackage{lastpage}
\jmlrheading{25}{2024}{1-\pageref{LastPage}}{6/23; Revised
	4/24}{10/24}{23-0805}{Alexander Davydov, Saber Jafarpour, Anton V. Proskurnikov, and Francesco Bullo}

\ShortHeadings{Non-Euclidean Monotone Operator Theory}{Davydov, Jafarpour, Proskurnikov, and Bullo}
\firstpageno{1}

\begin{document}

\title{Non-Euclidean Monotone Operator Theory and Applications}

\author{\name Alexander Davydov$\;\!$\thanks{The first two authors contributed equally} \email davydov@ucsb.edu \\
       \addr Center for Control, Dynamical Systems, and Computation\\
       University of California, Santa Barbara\\
       Santa Barbara, CA 93106-5070, USA
       \AND
       \name Saber Jafarpour$^{*}$ \email saber.jafarpour@colorado.edu \\
       \addr Department of Electrical, Computer, and Energy Engineering\\
       University of Colorado, Boulder \\
       Boulder, CO 80309-0020, USA
   	   \AND
   	   \name Anton V. Proskurnikov \email anton.p.1982@ieee.org \\
   	   \addr Department of Electronics and Telecommunications \\
   	   Politecnico di Torino\\
   	   Turin, Italy 10129
   	   \AND
   	   \name Francesco Bullo \email bullo@ucsb.edu \\
   	   \addr Center for Control, Dynamical Systems, and Computation\\
   	   University of California, Santa Barbara\\
   	   Santa Barbara, CA 93106-5070, USA
	}

\editor{Silvia Villa}

\maketitle

\begin{abstract}
While monotone operator theory is \change{often} studied on Hilbert spaces, many interesting problems in machine learning and optimization arise naturally in finite-dimensional vector spaces endowed with non-Euclidean norms, such as diagonally-weighted $\ell_{1}$ or $\ell_{\infty}$ norms.
This paper provides a natural generalization of monotone operator theory to finite-dimensional non-Euclidean spaces. The key tools are weak pairings and logarithmic norms. We show that the resolvent and reflected resolvent operators of non-Euclidean monotone mappings exhibit similar properties to their counterparts in Hilbert spaces. Furthermore, classical iterative methods and splitting methods for finding zeros of monotone operators are shown to converge in the non-Euclidean case. We apply our theory to equilibrium computation and Lipschitz constant estimation of recurrent neural networks, obtaining novel iterations and tighter upper bounds via forward-backward splitting.
\end{abstract}


\section{Introduction}

\textit{Problem description and motivation:}
Monotone operator theory is a fertile field of nonlinear functional analysis that extends the notion of monotone functions on $\real$ to mappings on Hilbert spaces. Monotone operator methods are widely used to
solve problems in machine learning~\citep{PLC-JCP:20b,EW-JZK:20}, data science~\citep{PLC-JCP:21}, optimization and control~\citep{AS:17,AB-ED-AS:19}, game theory~\citep{LP:20}, and systems analysis~\citep{TC-FF-RS:21}. A crucial part of this theory is the design of algorithms for computing zeros of monotone operators. This problem is central in convex optimization since (i) the subdifferential of any convex function is monotone and (ii) minimizing a convex function is equivalent to finding a zero of its subdifferential. To this end, there has been extensive research in the last decade in applying monotone operator methods to convex optimization; see, e.g.,~\citep{EKR-SB:16,PLC:18,EKR-WY:21}.

Existing monotone operator techniques are primarily based on inner-product spaces, while many problems are better-suited for analysis in more general normed spaces.  For instance robustness analysis of artificial neural networks in machine learning often requires the use of the $\ell_{\infty}$ norm for high-dimensional input data such as images~\citep{IJG-JS-CZ:15}. In distributed optimization, it is known that many natural conditions for the convergence of totally asynchronous algorithms are based upon contractions in an $\ell_{\infty}$ norm~\citep[Chapter~6, Section~3]{DPB-JNT:97}.

Motivated by problems in non-Euclidean spaces, we aim to extend monotone operator techniques for computing zeros of monotone operators to operators which are naturally ``monotone" with respect
to (w.r.t.) a non-Euclidean norm in a finite-dimensional space.


\textit{Literature review:} The literature on monotone operators dates back to Minty and Browder~\citep{GJM:62,FEB:67} and the connection to convex analysis was drawn upon by Minty and Rockafellar~\citep{GJM:64,RTR:66}. Since these foundational works, the theory of monotone operators over Hilbert spaces and its connection with convex optimization continues to expand, especially in the last decade~\citep{HHB-PLC:17,EKR-SB:16,EKR-WY:21,EKR-RH-WY:21}. Despite these connections between convex optimization and monotone operators, many problems in machine learning involve monotone operators beyond gradients of convex functions. Examples of such problems include generative adversarial networks, adversarially robust training of models, and training of models under fairness constraints. 
Instead of minimizing a convex function, to address these problems, one must solve for variational inequalities, monotone inclusions, and game-theoretic equilibria. 
In each of these more general cases, monotone operator theory has played an essential role in their analyses.  

In machine learning, monotone operators have been used in the training of generative adversarial networks~\citep{GG-HB-GV-PV-SLJ:19}, in
the design of novel neural network architectures~\citep{EW-JZK:20}, in the analysis of equilibrium behavior (infinite-depth limit) of neural networks~\citep{PLC-JCP:20b}, in the estimation of Lipschitz constants of neural networks~\citep{PLC-JCP:20,CP-EW-JZK:21},
and in normalizing flows~\citep{BA-CK-YH-HJK:22}.
Monotone operators have also been studied in the machine learning community in the context of variational inequality algorithms, stochastic monotone inclusions, and saddle-point problems; see e.g.~\citep{PB-FB:16,JD:20,XC-CS-CAG-JD:22,TP-PL-PP-OF-VC:22,XZ-WBH-ZY:22,AA-AB-YM:23,MIJ-TL-MZ:23,TY-MIJ-TC:23} for recent works in this direction. See also the recent survey~\citep{PLC-JCP:21} for applications in data science.  

The theory of dissipative and accretive operators on Banach spaces largely parallels the theory of monotone operators on Hilbert spaces~\citep{KD:85}. Despite these parallels, this theory has found far fewer direct applications to machine learning and data science; instead it is mainly applied for iterative solving integral equations and PDEs in $L_p$ spaces for $p \neq 2$ (see the book,~\citealt{CC:09}, for iterative methods). Moreover, many works in Banach spaces focus on spaces that have a uniformly smooth or uniformly convex structure, which finite-dimensional $\ell_{1}$ and $\ell_{\infty}$ spaces do not possess. In a similar vein, methods based on Bregman divergences utilize smoothness and strict convexity of the distance-generating convex functions~\citep{HHB-JMB-PLC:03}. Connections between logarithmic norms and dissipative and accretive operators may be found in~\citep{GS:86,GS:06}. 

A concept similar to a monotone operator in a Hilbert space is that of a contracting vector field in dynamical systems theory~\citep{WL-JJES:98}. 
If the metric \change{with respect to which the vector field is contracting} is the standard Euclidean distance, \change{the} vector field, $\OF$, is strongly infinitesimally contracting if and only if the negative vector field $-\OF$ is strongly monotone when thought of as on operator on $\real^{n}$. However, vector fields need not be contracting with respect to a Euclidean distance. Indeed, a vector field may be contracting w.r.t.\ a non-Euclidean norm but not a Euclidean one~\citep{ZA-EDS:14b}. Due to the connection between monotone operators and contracting vector fields, it is of interest to explore the properties of operators that may be thought of as monotone w.r.t.\ a non-Euclidean norm. In this spirit, preliminary connections
between contracting vector fields and monotone operators were made in~\citep{FB-PCV-AD-SJ:21e}.

\textit{Contributions:} Our contributions are as follows. First, to address the gap in applying monotone operator strategies to problems that arise in finite-dimensional non-Euclidean spaces, we propose a non-Euclidean monotone operator framework that is based on the theory of weak pairings~\citep{AD-SJ-FB:20o} and logarithmic norms. We use weak pairings as a substitute for inner products and we demonstrate that many classic results from monotone operator theory are applicable to its non-Euclidean counterpart. In particular, we show that the resolvent and reflected resolvent operators of a non-Euclidean monotone mapping exhibit properties similar to those arising in Hilbert spaces. To ensure that the resolvent and reflected resolvents have full domain, we prove an extension of the classic Minty-Browder theorem~\citep{GJM:62,FEB:67} in Theorem~\ref{thm:minty}. 


Second, leveraging the non-Euclidean monotone operator framework, we show that traditional iterative algorithms such as the forward step method and proximal point method can be used to compute zeros of non-Euclidean monotone mappings.
We provide convergence rate estimates for these iterative algorithms and the Cayley method in Theorems~\ref{thm:fwdstep},~\ref{thm:proximalconvergence}, and~\ref{thm:Cayley-method} and demonstrate that for diagonally-weighted $\ell_{1}$ and $\ell_{\infty}$ norms, they exhibit improved convergence rates compared to their Euclidean counterparts. Notably, we prove that for a Lipschitz mapping which is monotone w.r.t.\ a diagonally-weighted $\ell_{1}$ or $\ell_{\infty}$ norm, the forward step method is guaranteed to converge for a sufficiently small step size, whereas convergence cannot be guaranteed if the mapping is monotone with respect to a Euclidean norm.

Third, we study operator splitting methods for mappings which are monotone w.r.t.\ diagonally-weighted $\ell_{1}$ or $\ell_{\infty}$ norms. In Theorems~\ref{thm:fwd-backwd} and~\ref{thm:PR-DR} we prove that the forward-backward, Peaceman-Rachford, and Douglas-Rachford splitting algorithms are all guaranteed to converge, with some key differences compared to the classical theory. For instance, in the classical setting where two operators, $\OF$ and $\OG$, are monotone w.r.t.\ a Euclidean norm, the forward-backward splitting algorithm will only converge if $\OF$ is cocoercive. In contrast, when considering $\ell_{1}$ or $\ell_{\infty}$ norms, Lipschitzness of $\OF$ is sufficient for convergence.

Fourth, we present new insights into non-Euclidean properties of proximal operators and their impact on the study of special set-valued operator inclusions. Specifically, in Proposition~\ref{prop:nonexpansive}, we demonstrate that when $\OF$ is the subdifferential of a separable, proper, lower semicontinuous, convex function, its resolvent and reflected resolvent are nonexpansive with respect to an $\ell_{\infty}$ norm. To showcase the practical relevance of this result, we apply our non-Euclidean monotone operator theory to the equilibrium computation of a recurrent neural network (RNN). We extend the recent work of~\citep{SJ-AD-AVP-FB:21f} and show that our theory provides novel iterations and convergence criteria for RNN equilibrium computation.

Finally, we study the robustness of the RNN via its $\ell_{\infty}$ norm Lipschitz constant. In Theorem~\ref{thm:RNN-Lip}, we generalize the results from~\citep{CP-EW-JZK:21} to non-Euclidean norms and provide sharper estimates for the $\ell_{\infty}$ Lipschitz constant than were provided in the previous work~\citep{SJ-AD-AVP-FB:21f}.

A preliminary version of this work appeared in~\citep{AD-SJ-AVP-FB:21j}. Compared to this preliminary version, this version (i) provides novel theoretical results on the analysis of nonsmooth operators which are monotone with respect to general norms, (ii) proves a novel generalization of the classical Minty-Browder theorem for these non-Euclidean monotone mappings, (iii) study special classes of set-valued inclusions by providing novel non-Euclidean properties of proximal operators, (iv) includes a more comprehensive application to RNNs, allowing for more general activation functions and studies the robustness of the neural network by providing a tighter Lipschitz estimate, and (v) includes proofs of all technical results.
Finally, we provide further comparisons to monotone operator theory on Hilbert spaces. \change{Other prior work,~\citep{AD-SJ-FB:20o,AD-AVP-FB:22q}, focuses on continuous-time
	contracting dynamical systems with respect to non-Euclidean norms and their robustness properties. In contrast, this work instead uses weak pairings, developed in~\citep{AD-SJ-FB:20o}, to establish monotonicity properties of maps with respect to non-Euclidean norms and how we can find zeros of these maps using iterative methods. The prior works~\citep{AD-SJ-FB:20o,AD-AVP-FB:22q} do not consider these discrete-time iterations.}

\section{Preliminaries}
\subsection{Notation}
We let $\realnonnegative$ be the set of nonnegative real numbers and $\realpositive$ be the set of positive real numbers. For a set $\mathcal{S}$, let $2^{\mathcal{S}}$ denote its power set. For a complex number $z$, let $\mathrm{Re}(z)$ denote its real part. For a vector $\eta \in \real^{n}$, let $[\eta]$ denote the diagonal matrix satisfying $[\eta]_{ii} = \eta_i$, \change{where $\eta_{1},\dots,\eta_n$ are the components of $\eta$}. Given a matrix $A \in \real^{n \times n}$, let $\spectrum(A)$ denote its spectrum.
For a mapping \change{$\map{\OF}{\mcX}{\mcY}$ where $\mcX \subseteq \real^{n}, \mcY \subseteq \real^{m}$},
let $\dom(\OF)$ be its domain.
If $\OF$ is differentiable, let $\jac{\OF}(x) := \frac{\partial \OF(x)}{\partial x}$ denote its Jacobian evaluated at $x$. For a mapping $\map{\OF}{\real^{n}}{\real^{n}}$, let $\zero(\OF) := \setdef{x \in \real^{n}}{\OF(x) = \vectorzeros[n]}$ and $\fixed(\OF) = \setdef{x \in \real^{n}}{\OF(x) = x}$ be the sets of zeros of $\OF$ and fixed points of $\OF$, respectively. We let $\map{\Id}{\real^{n}}{\real^{n}}$ be the identity mapping and $I_n \in \real^{n \times n}$ be the $n \times n$ identity matrix. For a vector $y \in \real^{n}$, define the mapping $\map{\sign}{\real^{n}}{\{-1,0,1\}^{n}}$ by $(\sign(y))_i = y_i/|y_i|$ if $y_i \neq 0$ and zero otherwise. For a convex function $\map{f}{\real^{n}}{{]{-}\infty,+\infty]}}$, let $\partial f(x) = \setdef{g \in \real^{n}}{f(x)-f(y) \geq g^{\top} (x-y), \change{\forall} y \in \real^{n}}$ be its subdifferential at $x$ and, if $f$ is differentiable, $\nabla f(x)$ be its gradient at $x$.

\subsection{Logarithmic Norms and Weak Pairings}
Instrumental to the theory of non-Euclidean monotone operators are logarithmic norms (also referred to as matrix measures, or going forward, log norms) discovered by Dahlquist and Lozinskii in 1958~\citep{GD:58,SML:58}.
\begin{definition}[Logarithmic norm]\label{def:log-norm}
	Let $\|\cdot\|$ denote a norm on $\real^{n}$ and also the induced operator norm on the set of matrices $\real^{n \times n}$. The \emph{logarithmic norm} of a matrix $A \in \real^{n\times n}$ is
	\begin{equation}\label{eq:matrixmeasure}
		\mu(A) = \lim_{h \to 0^{+}} \frac{\|I_n + hA\| - 1}{h}.
	\end{equation}
\end{definition}
The log norm of a matrix $A$ is the one-sided directional derivative of the induced norm in the direction of $A$ evaluated at the identity matrix. It is well known that this limit exists for any norm and matrix~\change{\citep[Lemma~1]{SML:58}}.
Properties of log norms include positive homogeneity, subadditivity, convexity, and $\mathrm{Re}(\lambda) \leq \mu(A) \leq \|A\|$ for all $\lambda \in \spectrum(A)$, \change{\citep{SML:58}}.
Note that unlike the induced matrix norm, the log norm of a matrix may be negative. For more details on log norms,
see the influential survey~\citep{GS:06} and the recent monograph~\citep[Chapter~2]{FB:22-CTDS}.

We will be specifically interested in diagonally weighted $\ell_{1}$ and $\ell_{\infty}$
norms defined by
\begin{gather}
	\norm{x}{1,[\eta]} = \|[\eta]x\|_{1} =  \sum_{\change{i = 1}}^{\change{n}}\eta_i|x_i|
	\qquad\text{and}\qquad
	\norm{x}{\infty,[\eta]^{-1}} = \|[\eta]^{-1}x\|_{\infty} = \max_{\change{i \in \until{n}}}  \frac{1}{\eta_i}|x_i|.
\end{gather}
\change{The formulas for the corresponding induced norms and log norms are provided in \citep[Eq. (2.36)-(2.38)]{FB:23-CTDS} and are}
\begin{align*}
	\norm{A}{1,[\eta]} &= \max_{j\in\until{n}} \sum_{i=1}^{n} \frac{\eta_i}{\eta_j} |a_{ij}|,
	\quad
	&&\mu_{1,[\eta]}(A)  =
	\max_{j\in\until{n}} \Big( a_{jj} + \sum_{i=1,i\neq j}^{n}  |a_{ij}| \frac{\eta_i}{\eta_j}   \Big), \\
	\norm{A}{\infty,[\eta]^{-1}} &= \max_{i\in\until{n}} \sum_{j=1}^{n} \frac{\eta_j}{\eta_i}  |a_{ij}|,
	\quad
	&&\mu_{\infty,[\eta]^{-1}}(A) = \max_{i\in\until{n}} \Big( a_{ii} + \sum_{j=1,j\neq i}^{n}
	|a_{ij}| \frac{\eta_j}{\eta_i}\Big). \label{eq:inftyMatrixMeasure-appendix}
\end{align*}
We additionally review the notion of a weak pairing (WP) on
$\real^{n}$ from~\citep{AD-SJ-FB:20o} which generalizes inner products to
non-Euclidean spaces.
\begin{definition}[Weak pairing]
	A \emph{weak pairing} is a map
	$\WP{\cdot}{\cdot}: \real^{n} \times \real^{n} \to \real$ satisfying:
	\begin{enumerate}
		\item\label{WSIP1-appendix}(sub-additivity and continuity of first argument)
		$\WP{x_{1}+x_{2}}{y} \leq \WP{x_{1}}{y} + \WP{x_{2}}{y}$, for all
		$x_{1},x_{2},y \in \real^{n}$ and $\WP{\cdot}{\cdot}$ is continuous in its
		first argument,
		\item\label{WSIP3-appendix}(weak homogeneity)
		$\WP{\alpha x}{y}\! =\! \WP{x}{\alpha y} = \alpha\WP{x}{y}$ and
		$\WP{-x}{-y} = \WP{x}{y}$, for all
		$x,y \in \real^{n}, \alpha \geq 0$,
		\item\label{WSIP4-appendix}(positive definiteness) $\WP{x}{x} > 0$, for all
		$x \neq \vectorzeros[n],$
		\item\label{WSIP5-appendix}(Cauchy-Schwarz inequality)
		$|\WP{x}{y}| \leq \WP{x}{x}^{1/2}\WP{y}{y}^{1/2}$, for all
		$x, y \in \real^{n}.$
	\end{enumerate}
\end{definition}

For every norm $\|\cdot\|$ on $\real^{n}$, there exists a (possibly not
unique) compatible WP $\WP{\cdot}{\cdot}$ such that
$\|x\|^2=\WP{x}{x}$, for every $x\in \real^{n}$. If the norm is induced by
an inner product, the WP coincides with the inner product.
\begin{definition}[Deimling's inequality and curve norm derivative formula]\label{def:add-properties-WP}
	Let $\|\cdot\|$ be a norm on $\real^{n}$ with compatible WP $\WP{\cdot}{\cdot}$.
	\begin{enumerate}
		\item The WP $\WP{\cdot}{\cdot}$ satisfies~\emph{Deimling's inequality} if
		\begin{equation}
			\ds\WP{x}{y} \leq \|y\| \lim_{h \to 0^+} h^{-1}(\|y+hx\| - \|y\|), \qquad \text{ for all } x,y \in \real^{n}.
		\end{equation}
		
		\item The WP $\WP{\cdot}{\cdot}$ satisfies the \emph{curve norm derivative formula} if for all differentiable $\map{x}{{]a,b[}}{\real^{n}}$, $\|x(t)\|D^+\|x(t)\| = \WP{\dot{x}(t)}{x(t)}$ holds for almost every $t \in {]a,b[}$, where $D^+$ denotes the upper right Dini derivative.$\;\!$\footnote{The definition and properties of Dini derivatives are presented in~\citep{GG-SK:92}.}
	\end{enumerate}
	
\end{definition}

For every norm, there exists at least one WP that satisfies the properties in Definition~\ref{def:add-properties-WP}.$\;\!$\footnote{Indeed, given a norm, the map $\map{\WP{\cdot}{\cdot}}{\real^{n} \times \real^{n}}{\real}$ given by $\WP{x}{y} = \|y\|\lim_{h \to 0^+} h^{-1}(\|y+hx\| - \|y\|)$ defines a WP that satisfies all of these properties. For more discussions about properties of this pairing, we refer to~\citep[Section~13]{KD:85} and~\citep[Appendix A]{AD-SJ-FB:20o}.} Thus, going forward, we assume that WPs satisfy these additional properties. Indeed, due to Deimling's inequality, we have the following useful relationship between WPs and log norms.

\begin{lemma}[Lumer's equality,~{\citealt[Theorem~18]{AD-SJ-FB:20o}}]\label{lemma:Lumer}
	Let $\|\cdot\|$ be a norm on $\real^{n}$ with compatible WP $\WP{\cdot}{\cdot}$. Then for every $A \in \real^{n \times n}$,
	\begin{equation}
		\mu(A) = \sup_{\|x\| = 1} \WP{Ax}{x} = \sup_{x \neq \vectorzeros[n]} \frac{\WP{Ax}{x}}{\|x\|^2}.
	\end{equation}
\end{lemma}
We will focus on WPs corresponding to diagonally-weighted $\ell_{1}$ and $\ell_{\infty}$ norms.
Specifically, from~\citep[Table~III]{AD-SJ-FB:20o}, we introduce the
WPs
$\map{\WP{\cdot}{\cdot}_{1,[\eta]}, \WP{\cdot}{\cdot}_{\infty,[\eta]^{-1}}}{\real^{n}\times\real^{n}}{\real}$
defined by
\begin{equation}
	\label{def:WSIP-1+infty-appendix}
	\WP{x}{y}_{1,[\eta]} = \|y\|_{1,[\eta]}\sign(y)^\top [\eta]x
	\qquad\text{and}\qquad  \WP{x}{y}_{\infty,[\eta]^{-1}} = \max_{i \in I_{\infty}([\eta]^{-1}y)} \eta_i^{-2} y_ix_i.
\end{equation}
where $\Iinfty(x) = \setdef{i\in\until{n}}{|x_i|=\norm{x}{\infty}}$. One
can show that both of these WPs satisfy Deimling's inequality and the curve-norm derivative formula. Formulas for more general $\ell_p$ norms are available in~\citep{AD-SJ-FB:20o}.
\subsection{Contractions, Nonexpansive Maps, and Iterations}


\begin{definition}[Lipschitz continuity]
	Let $\|\cdot\|$ be a norm and $\map{\OF}{\real^{n}}{\real^{n}}$ be a mapping. $\OF$ is
	\emph{Lipschitz continuous} with constant $\ell \in \realnonnegative$ if
	\begin{equation}\label{eq:Lipschitz}
		\norm{\OF(x_{1})-\OF(x_{2})}{}\leq \ell \norm{x_{1}-x_{2}}{} \qquad \text{for all } x_{1},x_{2}\in\real^{n}.
	\end{equation}
	Moreover we define $\Lip(\OF)$ to be the minimal (or infimum) constant which satisfies~\eqref{eq:Lipschitz}.
\end{definition}

If two mappings $\map{\OF,\OG}{\real^{n}}{\real^{n}}$ are Lipschitz continuous w.r.t.\ the same norm, then the composition $\OF\circ \OG$ has Lipschitz constant $\Lip(\OF \circ \OG) \leq \Lip(\OF)\Lip(\OG)$.

\begin{definition}[One-sided Lipschitz mappings,~{\citealt{AD-SJ-FB:20o}}]
	Given a norm $\norm{\cdot}{}$ with compatible WP
	$\WP{\cdot}{\cdot}$, a
	map $\map{\OF}{\real^{n}}{\real^{n}}$ is \emph{one-sided Lipschitz}
	with constant $c\in\real$ if
	\begin{equation} \label{eq:osL=WSIP-appendix}
		\WP{\OF(x_{1})-\OF(x_{2})}{x_{1}-x_{2}}  \leq c
		\norm{x_{1}-x_{2}}{}^2   \qquad \text{for all } x_{1},x_{2}\in\real^{n}.
	\end{equation}
	Moreover we define $\osL(\OF)$ to be the minimal (or infimum) constant which satisfies~\eqref{eq:osL=WSIP-appendix}.
\end{definition}

\change{As was proved in~\citep[Theorem~27]{AD-SJ-FB:20o},} if $\map{\OF,\OG}{\real^{n}}{\real^{n}}$ are one-sided Lipschitz w.r.t.\ the same WP, then $\osL(\alpha\OF) = \alpha \osL(\OF)$\change{,} $\osL(\OF + \OG) \leq \osL(\OF) + \osL(\OG)$\change{, and $\osL(\OF + c \Id) = \osL(\OF) + c$} for all $\alpha \geq 0$\change{, $c \in \real$}. Note that (i) the one-sided
Lipschitz constant is upper bounded by the Lipschitz constant, (ii) a
Lipschitz \change{continuous} map is always one-sided Lipschitz, and (iii) the one-sided
Lipschitz constant may be negative. Moreover, if $\OF$ is locally Lipschitz \emph{continuous}, we have an alternative characterization of $\osL(\OF)$.

\begin{lemma}[$\osL(\OF)$ for locally Lipschitz \change{continuous} $\OF$,~{\citealt[Theorem~16]{AD-AVP-FB:22q}}]\label{lemma:osL-local-Lip}
	Suppose the map $\map{\OF}{\real^{n}}{\real^{n}}$ is locally Lipschitz \change{continuous}. Then $\OF$
	is one-sided Lipschitz
	with constant $c\in\real$ if and only if$\;$\footnote{Note that for locally Lipschitz \change{continuous} $\OF$, $\jac{\OF}(x)$ exists for almost every $x$ by Rademacher's theorem.}
	\begin{equation}
		\mu(\jac{\OF}(x)) \leq c  \qquad \text{for almost every } x\in\real^{n}.
	\end{equation}
\end{lemma}


\begin{definition}[Contractions and nonexpansive maps]
	Let $\map{\OT}{\real^{n}}{\real^{n}}$ be Lipschitz \change{continuous} w.r.t.\ $\|\cdot\|$. We say $\OT$ is a \emph{contraction} if $\Lip(\OT) < 1$, and $\OT$ is \emph{nonexpansive} if $\Lip(\OT) \leq 1$.
\end{definition}

\begin{definition}[Picard iteration]
	Let $\map{\OT}{\real^{n}}{\real^{n}}$ be a contraction w.r.t.\ a norm $\|\cdot\|$ with $\Lip(\OT) < 1$.
	The \emph{Picard iteration} applied to $\OT$ with initial condition $x^0$ defines the sequence $\{x^{k}\}_{k=0}^\infty$ by
	\begin{equation}
		x^{k+1} = \OT(x^{k}).
	\end{equation}
\end{definition}
By the Banach fixed-point theorem $\OT$ has a unique fixed point, $x^{*}$, and the Picard iteration applied to $\OT$ satisfy
$\|x^{k} - x^{*}\| \leq \Lip(\OT)^{k}\|x^0 - x^{*}\|,$
for any initial condition $x^0$.

If $\OT$ is nonexpansive with $\fixed(\OT) \neq \emptyset$, Picard iteration may fail to find a fixed point of $\OT$. Such situations can be addressed by the following iteration and convergence result, initially proved in~\citep{SI:76} and with rate given in~\citep{RC-JAS-JV:14}.


\begin{definition}[\KM iteration]
	Let $\map{\OT}{\real^{n}}{\real^{n}}$ be nonexpansive w.r.t.\ a norm $\|\cdot\|$. The \emph{\KM iteration}$\;\!$\footnote{The \KM iteration may be defined with step sizes $\theta_k \in {]0,1[}$ which vary for each iteration. In this document, we will only work with constant step sizes.} applied to $\OT$ with initial condition $x^0$ and $\theta \in {]0,1[}$ defines the sequence $\{x^{k}\}_{k=0}^\infty$ by
	\begin{equation}\label{eq:KMiterations}
		x^{k+1} = (1-\theta)x^{k} + \theta \OT(x^{k}).
	\end{equation}
\end{definition}

\begin{lemma}[\change{Asymptotic regularity and convergence} of \KM iteration, \change{\citealt{RC-JAS-JV:14,SI:76}}] \label{lemma:KMconvergence}
	Let $\map{\OT}{\real^{n}}{\real^{n}}$ be nonexpansive w.r.t.\ a norm $\|\cdot\|$ and consider the \KM iteration as in~\eqref{eq:KMiterations}. Suppose $\fixed(\OT) \neq \emptyset$. Then \change{for any initial condition $x^0$,}
	\begin{equation}
		\|x^{k} - \OT(x^{k})\| \leq \frac{2\inf_{x^{*} \in \fixed(\OT)}\|x^0 - x^{*}\|}{\sqrt{k\pi\theta(1-\theta)}} \change{= \bigO(1/\sqrt{k}).}
	\end{equation}
	\change{Moreover, the sequence of iterates, $\{x_k\}_{k=0}^\infty$, converges to a fixed point of $\OT$.}
\end{lemma}

\section{Non-Euclidean Monotone Operators}
\subsection{Definitions and Properties}

\begin{definition}[Non-Euclidean monotone mapping]\label{def:monotoneoperator}
	A mapping $\map{\OF}{\real^{n}}{\real^{n}}$ is strongly monotone with monotonicity parameter $c > 0$ w.r.t.\ a norm $\|\cdot\|$ on $\real^{n}$ if there exists a compatible WP $\WP{\cdot}{\cdot}$ and if for all $x,y \in \real^{n}$,
	\begin{equation}
		-\WP{-(\OF(x) - \OF(y))}{x - y} \geq c\|x - y\|^2.
	\end{equation}
	If the inequality holds with $c = 0$, we say $\OF$ is monotone w.r.t.\ $\|\cdot\|$.
\end{definition}
In the language of Banach spaces, such a function $\OF$ is called strongly accretive~\citep[Definition~8.10]{CC:09}. Note that Definition~\ref{def:monotoneoperator} is
equivalent to ${-}\osL(-\OF) \geq c$.

In the case of a Euclidean norm, the WP corresponds to the inner product and Definition~\ref{def:monotoneoperator} corresponds to the usual definition of a monotone operator as in \change{\citep{GJM:62} and~\citep[Definition~20.1]{HHB-PLC:17}.}


By properties of $\osL$, 
\change{if $\map{\OF,\OG}{\real^{n}}{\real^{n}}$ are both monotone w.r.t.\ the same norm (and WP), then ${-}\osL(-\OF - \OG) \geq {-}\osL(-\OF) - \osL(-\OG)$ and thus a sum of mappings which are monotone w.r.t.\ the same norm are monotone.}
Additionally, if $\OF$ is monotone with monotonicity parameter $c \geq 0$, then for any $\alpha \geq 0$, \change{${-}\osL(-\Id - \alpha \OF) = 1 - \alpha\osL(-\OF)$ and thus} $\Id + \alpha \OF$ is strongly monotone with monotonicity parameter $1 + \alpha c$.

\begin{remark}[Connection with contracting vector fields]\label{rmk:contraction}
	A mapping $\map{\OF}{\real^{n}}{\real^{n}}$ is strongly infinitesimally contracting with rate $c > 0$ w.r.t.\ a norm $\|\cdot\|$ on $\real^{n}$ provided $\osL(\OF) \leq -c$~\citep{AD-SJ-FB:20o}. If
	$c = 0$, we say $\OF$ is weakly infinitesimally contracting w.r.t.\ $\|\cdot\|$. Clearly $\OF$ is strongly monotone if and only if $-\OF$ is strongly infinitesimally contracting. \change{Vector fields which are strongly infinitesimally contracting w.r.t.\ a norm generate flows which are contracting with respect to the same norm. In the case of weakly infinitesimally contracting vector fields, their flows are nonexpansive.}
\end{remark}



\begin{lemma}[Monotonicity for locally Lipschitz \change{continuous} mappings]\label{lemma:monotone-local-Lip}
	Let $\map{\OF}{\real^{n}}{\real^{n}}$ be locally Lipschitz \change{continuous}. $\OF$ is (strongly) monotone with monotonicity parameter $c \geq 0$ w.r.t.\ a norm $\|\cdot\|$ if and only if $-\mu(-\jac{\OF}(x)) \geq c$ for almost every $x \in \real^{n}$.
\end{lemma}
\begin{proof}
	Lemma~\ref{lemma:monotone-local-Lip} is a straightforward application of Lemma~\ref{lemma:osL-local-Lip}.
\end{proof}

We can see the application of Lemma~\ref{lemma:monotone-local-Lip} more explicitly in the context of continuously differentiable monotone operators in Euclidean norms. To be specific, for an operator $\map{\OF}{\real^{n}}{\real^{n}}$, let $\|\cdot\|_{2}$ be the Euclidean norm with corresponding inner product $\inprod{\cdot}{\cdot}$.
	Then, following~\change{\citep{GJM:62}}, $\OF$ is monotone with respect to $\|\cdot\|_{2}$ if
	$$\inprod{\OF(x) - \OF(y)}{x-y} \geq 0, \qquad \text{for all } x,y \in \real^{n}.$$
	If $\OF$ is continuously differentiable, this condition is known to be equivalent to (see e.g.,~\citep{EKR-SB:16}) $\jac{\OF}(x) + \jac{\OF}(x)^\top \succeq 0$, or equivalently $-\mu_{2}(-\jac{\OF}(x)) \geq 0$ or $\frac{1}{2}\lambda_{\min}(\jac{\OF}(x) + \jac{\OF}(x)^\top) \geq 0$, \change{where $\mu_{2}(A) = \frac{1}{2}\lambda_{\max}(A+A^\top)$ is the log norm corresponding to the norm $\|\cdot\|_{2}$. This result} coincides with what was demonstrated in Lemma~\ref{lemma:monotone-local-Lip}.

\begin{example}
	An affine function $\OF(x) = Ax + b$ is monotone if and only if $-\mu(-A) \geq 0$ and strongly monotone with parameter $c$ if and only if $-\mu(-A) \geq c$. This condition implies that the spectrum of $A$ lies in the portion of the complex plane given by $\setdef{z \in \complex}{\mathrm{Re}(z) \geq c}$.
\end{example}

\subsection{Resolvent, Reflected Resolvents, Forward Step Operators, and Lipschitz Estimates}\label{subsec:key-operators}

Monotone operator theory transforms the problem of finding a zero of a monotone operator into finding a fixed point of a suitably defined operator. 
Monotone operator theory on Hilbert spaces studies the resolvent and reflected resolvent, operators dependent on the original operator, with fixed points corresponding to zeros of the original monotone operator.
In this subsection we study these same two operators and also the forward step operator in the context of operators which are monotone w.r.t.\ a non-Euclidean norm. In particular, we characterize the Lipschitz constants of these operators, first providing Lipschitz upper bounds for arbitrary norms and then specializing to diagonally-weighted $\ell_{1}$ and $\ell_{\infty}$ norms.

\begin{definition}[Resolvent and reflected resolvent]
	Let $\map{\OF}{\real^{n}}{\real^{n}}$ be a monotone mapping w.r.t.\ some norm. The resolvent of $\OF$ with parameter $\alpha > 0$ denoted by $\map{\OJ_{\alpha\OF}}{\dom(\OJ_{\alpha\OF})}{\real^{n}}$ and defined by
	\begin{equation}
		\OJ_{\alpha\OF} = (\Id + \alpha\OF)^{-1}.
	\end{equation}
	The reflected resolvent of $\OF$ with parameter $\alpha > 0$ is denoted by $\map{\OR_{\alpha\OF}}{\dom(\OR_{\alpha\OF})}{\real^{n}}$ and defined by
	\begin{equation}
		\OR_{\alpha\OF} = 2 \OJ_{\alpha\OF} - \Id.
	\end{equation}
\end{definition}

\begin{definition}[Forward step operator]
	Let $\map{\OF}{\real^{n}}{\real^{n}}$ be a mapping and $\alpha \in \real$. The forward step of $\OF$ with parameter $\alpha > 0$ is denoted by $\map{\OS_{\alpha\OF}}{\real^{n}}{\real^{n}}$ and defined by
	\begin{equation}
		\OS_{\alpha\OF} = \Id - \alpha \OF.
	\end{equation}
\end{definition}

Note that for any $\alpha > 0$, we have $\OF(x) = \vectorzeros[n]$ if and only if $x = \OJ_{\alpha\OF}(x) = \OR_{\alpha\OF}(x) = \OS_{\alpha\OF}(x)$, i.e., $\zero(\OF) = \fixed(\OJ_{\alpha\OF}) = \fixed(\OR_{\alpha\OF}) = \fixed(\OS_{\alpha\OF})$. \change{Note that under the assumption that $\OF$ is monotone, both $\OJ_{\alpha\OF}$ and $\OR_{\alpha\OF}$ are single-valued mappings.}

We have deliberately not been specific with the domains of the resolvent and reflected resolvent operators. As we will show in the following theorem, under mild assumptions (continuity and monotonicity), both of their domains are all of $\real^{n}$.

\begin{theorem}[A non-Euclidean Minty-Browder theorem]\label{thm:minty}
	Suppose $\map{\OF}{\real^{n}}{\real^{n}}$ is continuous and monotone. Then for every $\alpha > 0$, $\dom(\OJ_{\alpha\OF}) = \dom(\OR_{\alpha\OF}) = \real^{n}$.
\end{theorem}

\begin{proof}
	Note that $\dom(\OJ_{\alpha\OF}) = \real^{n}$ provided that for every $u \in \real^{n}$, there exists $x \in \real^{n}$ such that $(\Id + \alpha\OF)(x) = u$. To establish this fact, consider the differential equation
	\begin{equation}\label{eq:ode}
		\dot{x} = -x - \alpha\OF(x) + u =: G(x).
	\end{equation}
	Note that any equilibrium, $x^{*}$, of~\eqref{eq:ode} satisfies $(\Id + \alpha\OF)(x^{*}) = u$. Thus it suffices to show that the differential equation~\eqref{eq:ode} has an equilibrium. First we note that for all $x,y \in \real^{n}$,
	\begin{equation}
		\begin{aligned}
			\WP{G(x) - G(y)}{x-y} \leq \WP{-(x - y)}{x - y} + \alpha\WP{-(\OF(x)-\OF(y))}{x-y} \leq -\|x - y\|^2.
		\end{aligned}
	\end{equation}
	Thus, we conclude that $\osL(G) \leq -1$. In line with Remark~\ref{rmk:contraction}, we conclude that $G$ is strongly infinitesimally contracting which ensures uniqueness of solutions to~\eqref{eq:ode} (see~\citep[Theorem~31]{AD-SJ-FB:20o}). Let $\phi(t,x_0)$ denote the flow of the dynamics~\eqref{eq:ode} at time $t \geq 0$ from initial condition $x(0) = x_0$. Then by~\citep[Theorem~31]{AD-SJ-FB:20o}, we conclude that
	\begin{equation*}
		\|\phi(t,x_0) - \phi(t,y_0)\| \leq e^{-t}\|x_0 - y_0\|
	\end{equation*}
	for all $x_0,y_0 \in \real^{n}$ and for all $t \geq 0$. In other words, for a fixed $t > 0$, the map $x \mapsto \phi(t,x)$ is a contraction. By the Banach fixed point theorem, for $\tau > 0$, there exists unique $x^{*}$ such that $x^{*} = \phi(\tau,x^{*})$. 
	\change{Then either $x^{*}$ is an equilibrium point of~\eqref{eq:ode} or it is part of a periodic orbit with period $\tau$. If $x^{*}$ were part of a periodic orbit, then every other point on the periodic orbit would be a fixed point of $\phi(\tau,\cdot)$, contradicting the uniqueness of the fixed point from the Banach fixed point theorem. Thus, we conclude that $x^{*}$ is an equilibrium point of~\eqref{eq:ode}}
	and thus verifies $(\Id + \alpha\OF)(x^{*}) = u$. This proves that $\dom(\OJ_{\alpha\OF}) = \real^{n}$. The proof for $\OR_{\alpha\OF}$ is a consequence of $\dom(\OJ_{\alpha\OF}) = \real^{n}$.
\end{proof}

We have the following corollary about inverses of strongly monotone mappings.
\begin{corollary}[Lipschitz constants of inverses of strongly monotone operators]\label{lemma:inverse}
	Suppose $\OF: \real^{n} \to \real^{n}$ is continuous and strongly monotone with monotonicity parameter $c > 0$. Then $\map{\OF^{-1}}{\real^{n}}{\real^{n}}$ is a \change{Lipschitz} continuous mapping with Lipschitz constant estimate $\Lip(\OF^{-1}) \leq 1/c$.
\end{corollary}
\begin{proof}
	To see this fact, note that
	\begin{equation}\label{eq:aux}
		\|\OF(x) - \OF(y)\|\|x - y\| \geq -\WP{-(\OF(x) - \OF(y))}{x - y} \geq c\|x - y\|^2,
	\end{equation}
	where the left hand inequality is the Cauchy-Schwarz inequality for WPs. So if $\OF(x) = \OF(y)$, then necessarily $x = y$, which implies that $\OF^{-1}$ is a single-valued mapping. The fact that $\dom(\OF^{-1}) = \real^{n}$ follows the same argument as in Theorem~\ref{thm:minty} instead studying the differential equation $\dot{x} = -\OF(x)$. Choosing $u,v\in\real^{n}$ and substituting
	$x = \OF^{-1}(u), y = \OF^{-1}(v)$ into~\eqref{eq:aux}, we conclude
	\begin{equation}
		\|u - v\| \geq c\|x - y\| = c\|\OF^{-1}(u) - \OF^{-1}(v)\|,
	\end{equation}
	which shows that $\Lip(\OF^{-1}) \leq 1/c$.
\end{proof}

For each of $\OJ_{\alpha\OF}, \OR_{\alpha\OF}$ and, $\OS_{\alpha\OF}$ we have now established that each of their domains is all of $\real^{n}$ and that fixed points of these operators correspond to zeros of $\OF$. In order to compute zeros of $\OF$, we aim to provide estimates of the Lipschitz constants of $\OJ_{\alpha\OF}, \OR_{\alpha\OF},$ and $\OS_{\alpha\OF}$ as a function of $\alpha$ and the norm and show when these maps are either contractions or nonexpansive. The following lemmas characterize these Lipschitz estimates.

\begin{lemma}[Lipschitz estimates of the forward step operator]\label{lemma:fwdstepLip}
	Let $\map{\OF}{\real^{n}}{\real^{n}}$ be Lipschitz \change{continuous} w.r.t.\ the norm $\|\cdot\|$ with constant $\Lip(\OF) = \ell$.
	\begin{enumerate}[(i)]
		\item\label{generalnormfwdstep} Suppose $\OF$ is monotone w.r.t.\ $\|\cdot\|$ with monotonicity parameter $c \geq 0$, then
		\begin{equation}\label{eq:LipFwd-general}
			\Lip(\OS_{\alpha\OF}) \leq e^{-\alpha c} + e^{\alpha \ell} - 1 - \alpha \ell, \qquad \text{for all } \alpha > 0.
		\end{equation}
		\item\label{1infnormfwdstep} Alternatively suppose $\|\cdot\|$ is a diagonally weighted $\ell_{1}$ or $\ell_{\infty}$ norm and $\OF$ is monotone w.r.t.\ $\|\cdot\|$ with monotonicity parameter $c \geq 0$, then
		\begin{equation}\label{eq:LipFwd-1inf}
			\Lip(\OS_{\alpha\OF}) \leq 1 - \alpha c \leq 1, \qquad \text{for all } \alpha \in {\Big]0, \frac{1}{\diagL(\OF)}\Big]},
		\end{equation}
		where $\ds\diagL(\OF) := \sup_{x \in \real^{n} \setminus \Omega_{\OF}}\max_{i\in\until{n}} (\jac{\OF}(x))_{ii} \leq \ell$, where $\Omega_{\OF}$ is the measure zero set of points where $\OF$ is not differentiable.
	\end{enumerate}
\end{lemma}
\begin{proof}
	Regarding \change{item~(\ref{generalnormfwdstep})},
	we recall the inequality~\citep[pp. 14]{GD:58},~\citep[Prop. 2.1]{GS:06}
	\begin{equation}\label{eq:Coppel}
		\|e^{\alpha A}\| \leq e^{\alpha \mu(A)}, \qquad \text{ for all } \alpha \geq 0, A \in \real^{n \times n}.
	\end{equation}
	
	We additionally note that since $\OF$ is Lipschitz \change{continuous}, $\OS_{\alpha\OF}$ is as well and $\OS_{\alpha\OF}$ has  $\Lip(\OS_{\alpha\OF}) \leq L$ if and only if $\|\jac{\OS_{\alpha\OF}}(x)\| \leq L$ for almost every $x \in \real^{n}$. Also we have that $\jac{\OS_{\alpha\OF}}(x) = I_n - \alpha\jac{\OF}(x)$ everywhere it exists and that $\jac{\OF}(x)$ satisfies $-\mu(-\jac{\OF}(x)) \geq c$ and $\|\jac{\OF}(x)\| \leq \ell$. In what follows, when we write $\jac{\OF}(x)$, we mean for all $x$ for which the Jacobian exists.
	
	To derive an upper bound on $\|I_n - \alpha\jac{\OF}(x)\|$, we define
	\begin{equation*}
		S(x) := \sum_{i=2}^\infty \frac{(-\alpha)^i \jac{\OF}(x)^i}{i!} = e^{-\alpha \jac{\OF}(x)} - I_n + \alpha \jac{\OF}(x)
	\end{equation*}
	and it is straightforward to see that
	$\|S(x)\| \leq \sum_{i=2}^\infty \frac{\alpha^i \|\jac{\OF}(x)\|^i}{i!} \leq e^{\alpha\ell} - 1 - \alpha\ell.$
	Moreover,
	\begin{equation}
		\begin{aligned}
			\|I_n - \alpha \jac{\OF}(x)\| &\leq \|e^{-\alpha \jac{\OF}(x)}\| + \|S(x)\| \leq e^{\alpha\mu(-\jac{\OF}(x))} + e^{\alpha\ell} - 1 - \alpha\ell \\
			&\leq e^{-\alpha c}+e^{\alpha\ell} - 1 - \alpha\ell.
		\end{aligned}
	\end{equation}
	Since this bound holds for all $x$ for which $\jac{\OS_{\alpha\OF}}(x)$ exists, the result is proved.
	
	Regarding item~\change{(\ref{1infnormfwdstep})}, for every $x \in \real^{n}$ for which $\jac{\OF}(x)$ exists,
	\begin{align}
		\|I_n - \alpha \jac{\OF}(x)\|_{\infty,[\eta]^{-1}} &= \max_{i\in\until{n}} |1 - \alpha (\jac{\OF}(x))_{ii}| + \sum_{j=1,j\neq i}^n |{-}\alpha (\jac{\OF}(x))_{ij}|\frac{\eta_j}{\eta_i} \\
		&= \max_{i\in\until{n}} 1 - \alpha(\jac{\OF}(x))_{ii} + \sum_{j=1,j\neq i}^n |{-}\alpha (\jac{\OF}(x))_{ij}|\frac{\eta_j}{\eta_i} \label{eq:diagLbound}\\
		&= 1 + \alpha \mu(-\jac{\OF}(x)) \leq 1 - \alpha c \label{eq:norm-mu-trick},
	\end{align}
	where $\eqref{eq:diagLbound}$ holds because $0 < \alpha \leq \frac{1}{\diagL(\OF)}$ so that $1 - \alpha (\jac{\OF}(x))_{ii} \geq 0$ for all $x \in \real^{n}, i \in \until{n}$ and~\eqref{eq:norm-mu-trick} is due to the formula for $\mu_{\infty,[\eta]^{-1}}$. The proof for $\mu_{1,[\eta]}$ is analogous, replacing row sums by column sums, and is omitted.
\end{proof}
\begin{remark}
	If $c > 0$, then for small enough $\alpha > 0$, one can make the upper bound on $\Lip(\OS_{\alpha\OF})$ in~\eqref{eq:LipFwd-general} less than unity. In particular, one can show that minimizing the upper bound~\eqref{eq:LipFwd-general} yields the optimal step size $\subscr{\alpha}{opt} = \frac{1}{\ell}\ln(s(\gamma))$ and contraction factor $s(\gamma) + s(\gamma)^{-\gamma} -1 -\ln(s(\gamma))$, where $\gamma = c/\ell \leq 1$ and $s(\gamma)$ is the unique solution to the transcendental equation $s - 1 - \gamma s^{-\gamma} = 0$.
	
\end{remark}
\begin{remark}\label{rmk:nonexpansive-fwdstep}
	Note that for general norms, if $\OF$ is monotone, but not strongly monotone, then $\OS_{\alpha\OF}$ need not be nonexpansive for any $\alpha>0$. Indeed, consider $\OF(x) = \left(\begin{smallmatrix} 0 & 1 \\ -1 & 0 \end{smallmatrix}\right)x$, which is monotone w.r.t.\ the $\ell_{2}$ norm, but $\OS_{\alpha\OF}$ is not nonexpansive for any $\alpha > 0$. On the other hand, Lemma~\ref{lemma:fwdstepLip}(\ref{1infnormfwdstep}) implies that if $\OF$ is monotone w.r.t.\ a diagonally weighted $\ell_{1}$ or $\ell_{\infty}$ norm, then $\OS_{\alpha\OF}$ is nonexpansive for sufficiently small $\alpha$.
\end{remark}

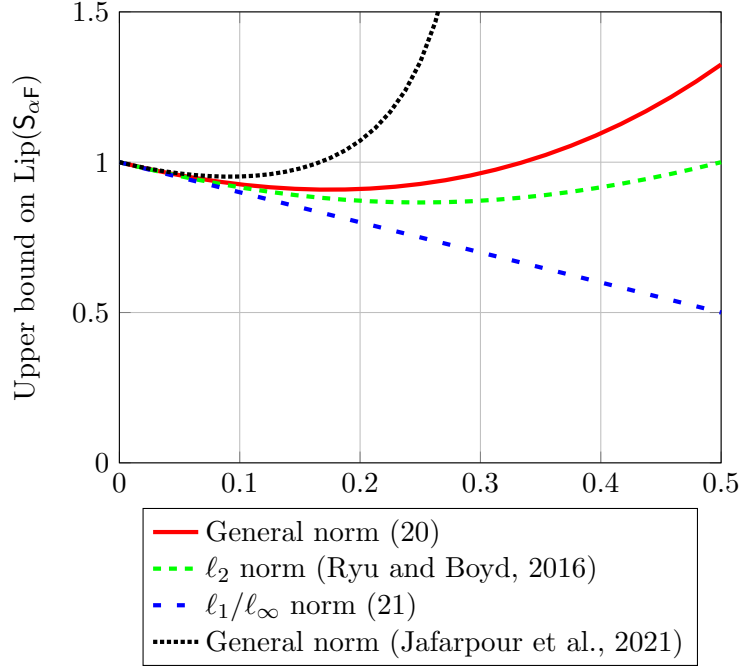
\begin{figure}[h]
	\centering
	\begin{tikzpicture}[scale = 1]
		\begin{axis}[width=8cm,height=6cm,scale only axis,
			xmin=0, xmax=0.5, ymin=0, ymax=1.5,
			grid=major,
			xlabel={$\alpha$}, ylabel={Upper bound on $\Lip(\OS_{\alpha\OF})$},
			legend style={at={(0.5,-0.2)},anchor=north},
			legend cell align={left}
			]
			
			\addplot[
			red, solid, ultra thick,
			domain=0:0.5
			] {exp(-x) + exp(2*x) - 1 - 2*x};
			\addlegendentry{General norm~\eqref{eq:LipFwd-general}}
			
			\addplot[
			green, dashed, ultra thick,
			domain=0:0.5
			] {sqrt(1-2*x + x*x*4)};
			\addlegendentry{$\ell_{2}$ norm~\citep{EKR-SB:16}}
			
			\addplot[
			blue, loosely dashed, ultra thick,
			domain=0:0.5
			] {1-x};
			\addlegendentry{$\ell_{1}/\ell_{\infty}$ norm~\eqref{eq:LipFwd-1inf}}
			
			\addplot[
			black, densely dotted, ultra thick,
			domain=0:0.3
			] {1/(1 + x - 4*x^2/(1-2*x))};
			\addlegendentry{General norm \citep{SJ-AD-AVP-FB:21f}}
			
		\end{axis}
	\end{tikzpicture}
	\caption{Plots of upper bounds of $\Lip(\OS_{\alpha\OF})$ with respect to different norms. We fix parameters $c = 1, \ell = 2$ and vary the choice of norm. The solid red curve corresponds to the Lipschitz estimate~\eqref{eq:LipFwd-general} for arbitrary norms, the densely dashed green curve corresponds to the estimate $\Lip(\OS_{\alpha\OF}) \leq \sqrt{1 - 2\alpha c + \alpha^2\ell^2}$ from~\citep[pp.~16]{EKR-SB:16} for the $\ell_{2}$ norm, the loosely dashed blue curve corresponds to the estimate~\eqref{eq:LipFwd-1inf} for diagonally-weighted $\ell_{1}/\ell_{\infty}$ norms which is valid on the interval $]0, \frac{1}{\diagL(\OF)}]$. Finally, the dotted black curve corresponds to the estimate $\Lip(\OS_{\alpha\OF}) \leq \big(1 + \alpha c - \frac{\alpha^2\ell^2}{1-\alpha\ell}\big)^{-1}$ previously established in~\citep[Theorem~1]{SJ-AD-AVP-FB:21f}. We see that the estimate~\eqref{eq:LipFwd-general} is a tighter estimate than the estimate from~\citep{SJ-AD-AVP-FB:21f} and that Lipschitz upper bounds are least conservative in the case of diagonally-weighted $\ell_{1}/\ell_{\infty}$ norms.} \label{fig:LipFwdStepEstimates}
\end{figure}

We plot the upper bounds on the estimates of $\Lip(\OS_{\alpha\OF})$ as a function of $\alpha$ and choice of norm for fixed parameters $c$ and $\ell$ in Figure~\ref{fig:LipFwdStepEstimates}.

\begin{lemma}[Lipschitz constant of the resolvent operator]\label{lemma:contractionResolvent}
	Suppose $\OF: \real^{n} \to \real^{n}$ is continuous and monotone with monotonicity parameter $c
	\geq 0$. Then,
	\begin{equation}
		\Lip(\OJ_{\alpha\OF}) \leq \frac{1}{1 + \alpha c}, \qquad \text{ for all } \alpha > 0.
	\end{equation}
\end{lemma}
\begin{proof}
	We observe that $\Id + \alpha\OF$ is
	strongly monotone with parameter $1 + \alpha c$. Then by
	Corollary~\ref{lemma:inverse}, the result holds.
\end{proof}

\begin{lemma}[Lipschitz constant of the reflected resolvent] \label{lemma:CayleyLip}
	Suppose $\OF: \real^{n} \to \real^{n}$ is Lipschitz \change{continuous} with constant $\ell$ w.r.t.\ a norm $\|\cdot\|$.
	\begin{enumerate}[(i)]
		\item\label{LipCayleyGeneral} Suppose $\OF$ is monotone w.r.t.\ $\|\cdot\|$ with monotonicity parameter $c
		\geq 0$. Then
		\begin{equation}
			\Lip(\OR_{\alpha\OF}) \leq \frac{e^{-\alpha c} + e^{\alpha \ell} - 1 - \alpha \ell}{1 + \alpha c}, \qquad \text{ for all } \alpha > 0.
		\end{equation}
		\item\label{LipCayley1inf} Alternatively suppose $\|\cdot\|$ is a diagonally weighted $\ell_{1}$ or $\ell_{\infty}$ norm. Moreover, suppose $\OF$ is monotone w.r.t.\ $\|\cdot\|$ with monotonicity parameter $c
		\geq 0$. Then,
		\begin{equation}
			\Lip(\OR_{\alpha\OF}) \leq \frac{1 - \alpha c}{1 + \alpha c} \leq 1, \qquad \text{ for all } \alpha \in {\Big]0,
				\frac{1}{\diagL(\OF)}\Big]}.
		\end{equation}
	\end{enumerate}
\end{lemma}
\begin{proof}
	Recall from~\citep[pp.~21]{EKR-SB:16} that since $\OF$ is monotone and continuous, we have that $\OR_{\alpha\OF} = \OS_{\alpha\OF} \circ \OJ_{\alpha\OF}$. Both results then follow from $\Lip(\OR_{\alpha\OF}) \leq \Lip(\OS_{\alpha\OF})\Lip(\OJ_{\alpha\OF})$ and the bounds on $\Lip(\OS_{\alpha\OF})$ from Lemma~\ref{lemma:fwdstepLip} and on $\Lip(\OJ_{\alpha\OF})$ from Lemma~\ref{lemma:contractionResolvent}.
\end{proof}

Lemma~\ref{lemma:CayleyLip} stands in striking contrast with results on monotone operators in Hilbert spaces which says that for any maximally monotone operator,$\;\!\!$\footnote{\change{Recall that in monotone operator theory on Hilbert spaces, a set-valued mapping $\map{\OF}{\real^{n}}{2^{\real^{n}}}$ is maximally monotone if it is monotone and there does not exist another monotone operator, $\OG$, whose graph properly contains the graph of $\OF$. See~\citep[Sec.~20.2]{HHB-PLC:17} for more details.}} $\OF$, the reflected resolvent of $\OF$ with parameter $\alpha > 0$ is nonexpansive for every $\alpha > 0$. Indeed, in the non-Euclidean case, this property cannot be recovered as is demonstrated in the following example.

\begin{example}\label{example:linear}
	Consider the linear mapping
	$\OF(x) = Ax = \left(\begin{smallmatrix} 2 & -2 \\ 1 & 1 \end{smallmatrix}\right)x.$
	$\OF$ is monotone w.r.t.\ the $\ell_{\infty}$ norm since
	$\ds-\mu_{\infty}(-A) = -\mu_{\infty}\left(\begin{smallmatrix} -2 & 2 \\ -1 & -1 \end{smallmatrix}\right) = 0.$
	For $\alpha = 2$, we compute
	\begin{equation*}
		\OJ_{\alpha\OF}(x) = \begin{pmatrix} 3/23 & 4/23 \\ -2/23 & 5/23 \end{pmatrix} x, \qquad \OR_{\alpha\OF}(x) = \begin{pmatrix} -17/23 & 8/23 \\ -4/23 & -13/23 \end{pmatrix} x.
	\end{equation*}
	Thus, $\Lip(\OJ_{\alpha\OF}) = 7/23$ and $\Lip(\OR_{\alpha\OF}) = 25/23$. In other words, for $\alpha = 2$, $\OJ_{\alpha\OF}$ is a contraction and $\OR_{\alpha\OF}$ is not nonexpansive.
\end{example}

Despite this key divergence from the classical theory, we will still be able to prove convergence of iterative algorithms involving the reflected resolvent under suitable assumptions on the parameter $\alpha > 0$.

\section{Finding Zeros of Non-Euclidean Monotone Operators}


For a mapping $\map{\OF}{\real^{n}}{\real^{n}}$ which is continuous and monotone, consider the problem of finding an $x \in \real^{n}$ that satisfies
\begin{equation}
	\OF(x) = \vectorzeros[n].
\end{equation}
Without further assumptions on $\OF$, this problem may have no solutions or nonunique solutions. First we provide a preliminary sufficient condition for existence and uniqueness of a solution.

\begin{lemma}[Uniqueness of zeros of strongly monotone maps]
	Suppose $\map{\OF}{\real^{n}}{\real^{n}}$ is continuous and strongly monotone. Then $\zero(\OF)$ is a singleton.
\end{lemma}
\begin{proof}
	We have that $\zero(\OF) = \fixed(\OJ_{\alpha\OF})$ for $\alpha > 0$. By Lemma~\ref{lemma:contractionResolvent}, we have that $\Lip(\OJ_{\alpha\OF}) \leq 1/(1+\alpha c) < 1$, where $c>0$ is the monotonicity parameter of $\OF$. Then by the Banach fixed point theorem, $\OJ_{\alpha\OF}$ has a unique fixed point and thus $\zero(\OF)$ is a singleton.
\end{proof}
Alternatively, if $\OF$ is continuous and monotone, then we study fixed points of the nonexpansive map $\OJ_{\alpha\OF}$, which may or may not exist and may or may not be unique. In what follows, we will study the case where it is known a priori that zeros of $\OF$ exist but need not be unique.

We show that the most known algorithms for finding zeros of monotone operators on Hilbert spaces (see, e.g.,~\citep{EKR-SB:16})
can be generalized to non-Euclidean monotone operators using our framework and, furthermore, explicitly estimate the convergence rate of these methods.

\subsection{The Forward Step Method}

\begin{algo}[Forward step method]\label{alg:fwdstep}
	The \emph{forward step method} corresponds to the fixed point iteration
	\begin{equation}\label{eq:forwardstep}
		x^{k+1} = \OS_{\alpha\OF}(x^{k}) = x^{k} - \alpha \OF(x^{k}).
	\end{equation}
\end{algo}

\begin{theorem}[Convergence guarantees for the forward step method]\label{thm:fwdstep}
	Let $\map{\OF}{\real^{n}}{\real^{n}}$ is Lipschitz \change{continuous} with constant $\ell$ w.r.t.\ a norm $\|\cdot\|$ and let $x^0 \in \real^{n}$ be arbitrary.
	\begin{enumerate}[(i)]
		\item\label{fwdstepgeneral} Suppose $\OF$ is strongly monotone w.r.t.\ $\|\cdot\|$ with monotonicity parameter $c > 0$. Then the iteration~\eqref{eq:forwardstep} \change{converges} to the unique zero, $x^{*}$, of $\OF$ for every $\alpha \in {]0, \alpha^{*}[}$. Moreover, for every $k \in \mathbb{Z}_{\geq 0}$,
		\begin{equation*}
			\|x^{k+1} - x^{*}\| \leq (e^{-\alpha c} + e^{\alpha \ell} - 1 - \alpha\ell)\|x^{k}-x^{*}\|,
		\end{equation*}
		where $\alpha^{*}$ is the unique positive value of $\alpha$ that satisfies $e^{-\alpha^{*} c} + e^{\alpha^{*} \ell} = 2 + \alpha^{*}\ell$.
		\item\label{fwdstep1inf} Alternatively suppose $\|\cdot\|$ is a diagonally-weighted $\ell_{1}$ or $\ell_{\infty}$ norm and $\OF$ is strongly monotone w.r.t.\ $\|\cdot\|$ with monotonicity parameter $c > 0$. Then the iteration~\eqref{eq:forwardstep} converges to the unique zero, $x^{*}$, of $\OF$ for every $\alpha \in {]0,\frac{1}{\diagL(\OF)}]}$. Moreover, for every $k \in \mathbb{Z}_{\geq 0}$,
		\begin{equation*}
			\|x^{k+1} - x^{*}\| \leq (1 - \alpha c)\|x^{k}-x^{*}\|,
		\end{equation*}
		with the convergence rate optimized at $\alpha = 1/\diagL(\OF)$.
		\item\label{fwdstepmonotone} Alternatively suppose $\|\cdot\|$ is a diagonally weighted $\ell_{1}$ or $\ell_{\infty}$ norm and $\OF$ is monotone w.r.t.\ $\|\cdot\|$. Then $\zero(\OF) \neq \emptyset$ implies the iteration~\eqref{eq:forwardstep} converges to an element of $\zero(\OF)$ for every $\alpha \in {]0,\frac{1}{\diagL(\OF)}[}.$
	\end{enumerate}
\end{theorem}

\begin{proof}
	Regarding statement~(\ref{fwdstepgeneral}), from Lemma~\ref{lemma:fwdstepLip}(\ref{fwdstepgeneral}), we have that $\Lip(\OS_{\alpha\OF}) \leq e^{-\alpha c} + e^{\alpha \ell} - 1 - \alpha \ell$. It is straightforward to compute that at $\alpha = \alpha^{*}$, $\Lip(\OS_{\alpha\OF}) \leq 1$ and for $\alpha \in {]0,\alpha^{*}[}$ we have that $\Lip(\OS_{\alpha\OF}) < 1$. Thus, $\OS_{\alpha\OF}$ is a contraction and fixed points of $\OS_{\alpha\OF}$ correspond to zeros of $\OF$. Then by the Banach fixed point theorem, the result follows.
	
	Regarding statement~(\ref{fwdstep1inf}), Lemma~\ref{lemma:fwdstepLip}(\ref{fwdstep1inf}) implies that $\Lip(\OS_{\alpha\OF}) = 1-\alpha c < 1$ for all $\alpha \in {]0,1/\diagL(\OF)]}$. The result is then a consequence of the Banach fixed point theorem.
	
	Regarding statement~(\ref{fwdstepmonotone}), since $\OF$ is monotone w.r.t.\ a diagonally weighted $\ell_{1}$ or $\ell_{\infty}$ norm, $\OS_{\alpha\OF}$ is nonexpansive for $\alpha \in {]0,1/\diagL(\OF)]}$ by Lemma~\ref{lemma:fwdstepLip}(\ref{1infnormfwdstep}). Moreover, for every $\alpha \in {]0,1/\diagL(\OF)[}$, there exists $\theta \in {]0,1[}$ such that
	$\OS_{\alpha\OF} = (1-\theta) \Id + \theta\OS_{\tilde{\alpha}\OF},$
	for some $\tilde{\alpha} \in {]0,1/\diagL(\OF)]}$. Therefore the iteration~\eqref{eq:forwardstep} is the \KM iteration of the nonexpansive operator $\OS_{\tilde{\alpha}\OF}$ and Lemma~\ref{lemma:KMconvergence} implies the result.
\end{proof}

Note that Theorem~\ref{thm:fwdstep}(\ref{fwdstepmonotone}) is a direct \change{consequence} of the fact that the forward step operator is nonexpansive for suitable $\alpha > 0$ when the mapping is monotone w.r.t.\ a diagonally-weighted $\ell_{1}$ or $\ell_{\infty}$ norm, a fact which need not hold when the mapping is monotone w.r.t.\ a different norm, e.g., a Hilbert one. \change{See the relevant discussion in Remark~\ref{rmk:nonexpansive-fwdstep} for an example of a mapping, $\OF$, which is monotone with respect to the $\ell_{2}$ norm but $\OS_{\alpha\OF}$ is not nonexpansive for any $\alpha > 0$.}
\subsection{The Proximal Point Method}

\begin{algo}[Proximal point method]\label{alg:proximal}
	The \emph{proximal point method} corresponds to the fixed point iteration
	\begin{equation}\label{eq:resolventIteration}
		x^{k+1} = \OJ_{\alpha\OF}(x^{k}) = (\Id + \alpha\OF)^{-1}(x^{k}).
	\end{equation}
\end{algo}

\begin{theorem}[Convergence guarantees for the proximal point method]\label{thm:proximalconvergence}
	Suppose $\map{\OF}{\real^{n}}{\real^{n}}$ is continuous and let $x^0 \in \real^{n}$ be arbitrary.
	\begin{enumerate}[(i)]
		\item\label{proximalStrong} Suppose $\OF$ is strongly monotone w.r.t.\ a norm $\|\cdot\|$ with monotonicity parameter $c > 0$. Then the iteration~\eqref{eq:resolventIteration} converges to the unique zero, $x^{*}$, of $\OF$ for every $\alpha \in {]0, \infty[}$. Moreover, for every $k \in \mathbb{Z}_{\geq 0}$, the iteration satisfies
		\begin{equation*}
			\|x^{k+1} - x^{*}\| \leq \frac{1}{1+\alpha c}\|x^{k}-x^{*}\|.
		\end{equation*}
		\item\label{proximalLipschitz} Alternatively suppose $\OF$ is monotone and globally Lipschitz \change{continuous} w.r.t.\ a diagonally weighted $\ell_{1}$ or $\ell_{\infty}$ norm $\|\cdot\|$ and $\diagL(\OF) \neq 0$. Then if $\zero(\OF) \neq \emptyset$, the iteration~\eqref{eq:resolventIteration} converges to an element of $\zero(\OF)$ for every $\alpha \in {]0,\infty[}$.
	\end{enumerate}
\end{theorem}
\begin{proof}
	Regarding statement~(\ref{proximalStrong}), Lemma~\ref{lemma:contractionResolvent} provides the Lipschitz estimate $\Lip(\OJ_{\alpha\OF}) \leq \frac{1}{1+\alpha c} < 1$ for all $\alpha > 0$. Thus $\OJ_{\alpha\OF}$ is a contraction and since fixed points of $\OJ_{\alpha\OF}$ correspond to zeros of $\OF$, the Banach fixed point theorem implies the result.
	
	Regarding statement~(\ref{proximalLipschitz}), we will demonstrate that the iteration~\eqref{eq:resolventIteration} is a \KM iteration of a suitably-defined nonexpansive mapping. To this end, let $\theta \in {]0,1[}$ be arbitrary and consider the auxiliary mapping $\map{\overline{\OR}_{\alpha\OF}^{\theta}}{\real^{n}}{\real^{n}}$ given by $\overline{\OR}_{\alpha\OF}^{\theta} := \frac{\OJ_{\alpha\OF}}{\theta} - \frac{1 - \theta}{\theta} \Id.$
	Then it is straightforward to compute
	\begin{equation*}
		\begin{aligned}
			\overline{\OR}_{\alpha\OF}^{\theta} &= \frac{(\Id + \alpha \OF)^{-1}}{\theta} - \frac{1 - \theta}{\theta}(\Id + \alpha \OF) \circ (\Id + \alpha \OF)^{-1} \\
			&= \Big(\frac{\Id}{\theta} - \frac{1 - \theta}{\theta}(\Id + \alpha\OF)\Big) \circ (\Id + \alpha \OF)^{-1} = \Big(\Id - \frac{(1 - \theta)\alpha}{\theta}\OF\Big) \circ \OJ_{\alpha\OF} = \OS_{\frac{1-\theta}{\theta}\alpha \OF} \circ \OJ_{\alpha\OF}.
		\end{aligned}
	\end{equation*}
	Moreover, $\OJ_{\alpha\OF}$ is nonexpansive by Lemma~\ref{lemma:contractionResolvent}, and by Lemma~\ref{lemma:fwdstepLip}(\ref{1infnormfwdstep}),
	$\ds\Lip(\OS_{\frac{1-\theta}{\theta}\alpha \OF}) \leq 1,$ for all $\alpha \in {]0, \frac{1 - \theta}{\theta\diagL(\OF)}]}.$
	We conclude that $\Lip(\overline{\OR}_{\alpha\OF}^{\theta}) \leq \Lip(\OS_{\frac{1-\theta}{\theta}\alpha \OF})\Lip(\OJ_{\alpha\OF}) \leq 1$ for $\alpha \in {]0, \frac{1 - \theta}{\theta\diagL(\OF)}]}$ which implies that $\overline{\OR}_{\OF}^{\theta}$ is nonexpansive for all $\alpha$ in this range.
	
	Let $\alpha > 0$ be arbitrary. Then for any\footnote{Note that $\frac{1}{1 + \alpha\diagL(\OF)} \in {]0,1[}$ holds under the assumption $\diagL(\OF) \neq 0$ since $\diagL(\OF) \geq 0$ for any monotone $\OF$.}
	$\theta \leq \frac{1}{1 + \alpha\diagL(\OF)} \in {]0,1[},$
	we have that
	$\OJ_{\alpha\OF} = (1 - \theta)\Id + \theta \overline{\OR}_{\alpha\OF}^{\theta},$
	and $\overline{\OR}_{\OF}^{\theta}$ is nonexpansive since $\alpha \in {]0, \frac{1 - \theta}{\theta\diagL(\OF)}]}$. Thus, the iteration~\eqref{eq:resolventIteration} is the \KM iteration for $\overline{\OR}_{\alpha\OF}^{\theta}$ and the result follows from Lemma~\ref{lemma:KMconvergence}.
\end{proof}

\begin{remark}
	Theorem~\ref{thm:proximalconvergence}(\ref{proximalLipschitz}) is an analog of the classical result in monotone operator theory on Hilbert spaces which states that the resolvent of a maximally monotone operator is firmly nonexpansive~\change{\citep{GJM:62} and}~\citep[Prop.~23.8]{HHB-PLC:17}. 
	This firm nonexpansiveness is a consequence of the fact that the reflected resolvent of a maximally monotone operator with respect to a Euclidean norm is always nonexpansive and $\OJ_{\alpha\OF} = \frac{1}{2}\Id + \frac{1}{2}\OR_{\alpha\OF}$. Note that this property need not hold when $\OF$ is monotone with respect to a non-Euclidean norm but we are able to show that in the case of diagonally-weighted $\ell_{1}/\ell_{\infty}$ norms, a similar result holds.
\end{remark}

\subsection{The Cayley Method}

\begin{algo}[Cayley method]\label{alg:Cayley}
	The \emph{Cayley method} corresponds to the iteration
	\begin{equation}\label{eq:CayleyIteration}
		x^{k+1} = \OR_{\alpha\OF}(x^{k}) = 2(\Id + \alpha\OF)^{-1}(x^{k}) - x^{k}.
	\end{equation}
\end{algo}	

\begin{theorem}[Convergence guarantees for the Cayley method]\label{thm:Cayley-method}
	Suppose $\map{\OF}{\real^{n}}{\real^{n}}$ is Lipschitz \change{continuous} with constant $\ell$ w.r.t.\ a norm $\|\cdot\|$ and let $x^0 \in \real^{n}$ be arbitrary.
	\begin{enumerate}[(i)]
		\item\label{Cayleygeneral} Suppose $\OF$ is strongly monotone w.r.t.\ $\|\cdot\|$ with monotonicity parameter $c > 0$. Then the iteration~\eqref{eq:CayleyIteration} converges to the unique zero, $x^{*}$, of $\OF$ for sufficiently small $\alpha > 0$. Moreover, for every $k \in \mathbb{Z}_{\geq 0}$, the iteration satisfies
		\begin{equation*}
			\|x^{k+1} - x^{*}\| \leq \frac{e^{-\alpha c} + e^{\alpha\ell} - 1 - \alpha\ell}{1 + \alpha c}\|x^{k}-x^{*}\|.
		\end{equation*}
		\item\label{Cayley1inf} Alternatively suppose $\|\cdot\|$ is a diagonally weighted $\ell_{1}$ or $\ell_{\infty}$ norm and $\OF$ is strongly monotone w.r.t.\ $\|\cdot\|$ with monotonicity parameter $c > 0$. Then the iteration~\eqref{eq:CayleyIteration} converges to the unique zero, $x^{*}$, of $\OF$ for every $\alpha \in {]0,\frac{1}{\diagL(\OF)}]}$. Moreover, for every $k \in \mathbb{Z}_{\geq 0}$, the iteration satisfies
		\begin{equation*}
			\|x^{k+1} - x^{*}\| \leq \frac{1 - \alpha c}{1 + \alpha c}\|x^{k}-x^{*}\|,
		\end{equation*}
		with the convergence rate being optimized at $\alpha = 1/\diagL(\OF)$.
		\item\label{CayleyAveraged} Alternatively suppose $\|\cdot\|$ is a diagonally weighted $\ell_{1}$ or $\ell_{\infty}$ norm and $\OF$ is monotone w.r.t.\ $\|\cdot\|$. Then if $\zero(\OF) \neq \emptyset$, the \KM iteration with $\theta = 1/2$
		\begin{equation*}
			x^{k+1} = \frac{1}{2}x^{k} + \frac{1}{2}\OR_{\alpha\OF}(x^{k})
		\end{equation*}
		correspond to the proximal point iteration~\eqref{eq:resolventIteration}, which is guaranteed to converge to an element of $\zero(\OF)$ for every $\alpha \in {]0,\infty[}$.
	\end{enumerate}
\end{theorem}
\begin{proof}
	Regarding statement~(\ref{Cayleygeneral}), from Lemma~\ref{lemma:CayleyLip}(\ref{LipCayleyGeneral}), we have that $\Lip(\OR_{\alpha\OF}) \leq (e^{-\alpha c} + e^{\alpha\ell} - 1 - \alpha\ell)/(1+\alpha c)$ which is less than unity for small enough $\alpha > 0$. Thus, for small enough $\alpha$, $\OR_{\alpha\OF}$ is a contraction and fixed points of $\OR_{\alpha\OF}$ correspond to zeros of $\OF$. Thus, by the Banach fixed point theorem, the result follows.
	
	Regarding statement~(\ref{Cayley1inf}), Lemma~\ref{lemma:CayleyLip}(\ref{LipCayley1inf}) implies that $\Lip(\OR_{\alpha\OF}) \leq (1-\alpha c)/(1+\alpha c) < 1$ for $\alpha \in {]0,1/\diagL(\OF)]}$. The result is then a consequence of the Banach fixed point theorem.
	
	Statement~(\ref{CayleyAveraged}) holds since
	\change{$\frac{1}{2}\Id + \frac{1}{2}(2\OJ_{\alpha\OF} - \Id) = \OJ_{\alpha\OF}$},
	and convergence follows by Theorem~\ref{thm:proximalconvergence}(ii) since $\zero(\OF) \neq \emptyset$.
\end{proof}


\begin{table}[]
	\centering
	\resizebox{.999\textwidth}{!}{	\begin{tabular}{|c|c|c|c|c|c|c|}
			\hline
			\multirow{3}{*}{Algorithm, Iterated map} & \multicolumn{6}{c|}{$\OF$ strongly monotone and globally Lipschitz continuous} \\ \cline{2-7}
			& \multicolumn{2}{c|}{$\ell_{2}$} & \multicolumn{2}{c|}{General norm} & \multicolumn{2}{c|}{Diagonally weighted $\ell_{1}$ or $\ell_{\infty}$} \\ \cline{2-7}
			& $\alpha$ range & Optimal $\Lip$ & $\alpha$ range & Optimal $\Lip$ & $\alpha$ range & Optimal $\Lip$ \\ \hline
			& & & & & & \\[-0.3em]
			Forward step, $\OS_{\alpha\OF}$ & $\ds{\Big]0, \frac{2c}{\ell^2}\Big[}$ & $\ds1-\frac{1}{2\kappa^2} + \bigO\Big(\frac{1}{\kappa^3}\Big)$ & $\ds{]0, \alpha^{*}[}$ & $\ds1-\frac{1}{2\kappa^2} + \bigO\Big(\frac{1}{\kappa^3}\Big)$ & $\ds{\Big]0, \frac{1}{\diagL(\OF)}\Big]}$ & $\ds1-\frac{1}{\kappa_{\infty}}$ \\ 
			& & & & & & \\[-0.3em]
			Proximal point, $\OJ_{\alpha\OF}$ & ${]0, \infty[}$ & A.S. & ${]0, \infty[}$ & A.S. & ${]0, \infty[}$ & A.S. \\ 
			& & & & & & \\[-0.3em]
			Cayley, $\OR_{\alpha\OF}$ & ${]0, \infty[}$ & $\ds1-\frac{1}{2\kappa}+\bigO\Big(\frac{1}{\kappa^2}\Big)$ & $\ds{]0, \alpha^{*}[}$ & $\ds1-\frac{2}{\kappa^2} + \bigO\Big(\frac{1}{\kappa^3}\Big)$ & $\ds{\Big]0, \frac{1}{\diagL(\OF)}\Big]}$ & $\ds1-\frac{2}{\kappa_{\infty}}+ \bigO\Big(\frac{1}{\kappa_{\infty}^2}\Big)$ \\ \hline
	\end{tabular}}
	\caption{Table of step size ranges and Lipschitz constants for three algorithms for finding zeros of monotone operators with respect to arbitrary norms. For $\OF$ strongly monotone \change{and Lipschitz continuous}, let $c$ be its monotonicity parameter (with respect to the appropriate norm), $\ell$ its appropriate Lipschitz constant, and $\diagL(\OF) := \sup_{x \in \real^{n} \setminus \Omega_{\OF}}\max_{i\in\until{n}} (\jac{\OF}(x))_{ii} \leq \ell$. Additionally, $\kappa := \ell/c \geq 1,$ $\kappa_{\infty} := \diagL(\OF)/c \in [1,\kappa],$ and $\alpha^{*}$ is the unique positive solution to $e^{-\alpha^{*} c} + e^{\alpha^{*} \ell} = 2 + \alpha^{*}\ell$. \emph{A.S.} means the Lipschitz constant can be made arbitrarily small. We do not assume that the strongly monotone $\OF$ is the gradient of a strongly convex function.
	}\label{table:Lipschitz-stepsizes}
\end{table}

Table~\ref{table:Lipschitz-stepsizes} summarizes and compares the range of step sizes and Lipschitz estimates as provided by the classical monotone operator theory for the $\ell_{2}$ norm~\citep[pp.~16 and~20]{EKR-SB:16} and by Theorems~\ref{thm:fwdstep},~\ref{thm:proximalconvergence}, and~\ref{thm:Cayley-method} for general and diagonally-weighted $\ell_{1}/\ell_{\infty}$ norms. 

\section{Finding Zeros of a Sum of Non-Euclidean Monotone Operators}
In many instances, one may wish to execute the proximal point method, Algorithm~\ref{alg:proximal}, to compute a zero of a continuous monotone mapping $\map{\ON}{\real^{n}}{\real^{n}}$. However, the implementation of the iteration~\eqref{eq:resolventIteration} may be hindered by the difficulty in evaluating $\OJ_{\alpha \ON}$. To remedy this issue, it is often assumed that $\ON$ can be expressed as the sum of two monotone mappings $\OF$ and $\OG$ where $\OJ_{\alpha\OG}$ may be easy to compute and $\OF$ satisfies some regularity condition. Alternatively, in some situations, decomposing $\ON = \OF+ \OG$ and finding $x \in \real^{n}$ such that $(\OF + \OG)(x) = \vectorzeros[n]$ provides additional flexibility in choice of algorithm and may improve convergence rates.

Motivated by the above, we consider the problem of finding an $x \in \real^{n}$ such that
\begin{equation}
	(\OF+\OG)(x) = \vectorzeros[n],
\end{equation}
where $\map{\OF,\OG}{\real^{n}}{\real^{n}}$ are continuous and monotone w.r.t.\ a diagonally weighted $\ell_{1}$ or $\ell_{\infty}$ norm.\footnote{The results that follow can also be generalized to arbitrary norms using the Lipschitz estimates derived for $\OJ_{\alpha\OF}, \OR_{\alpha\OF},$ and $\OS_{\alpha\OF}$ in Section~\ref{subsec:key-operators}.} In particular, we focus on the forward-backward, Peaceman-Rachford, and Douglas-Rachford splitting algorithms. For some extensions of the theory to set-valued mappings, we refer to Section~\ref{subsec:set-valued}.

\subsection{Forward-Backward Splitting}

\begin{algo}[Forward-backward splitting]\label{alg:fwd-backwd}
	Assume $\alpha > 0$. Then in~\citep[Section~7.1]{EKR-SB:16} it is shown that
	\begin{equation*}
		(\OF + \OG)(x) = \vectorzeros[n] \quad \iff x = (\OJ_{\alpha\OG}\circ \OS_{\alpha\OF})(x).
	\end{equation*}
	The \emph{forward-backward splitting method} corresponds to the fixed point iteration
	\begin{equation}\label{eq:fwdbackward}
		x^{k+1} = \OJ_{\alpha\OG}(x^{k} - \alpha\OF(x^{k})).
	\end{equation}
	If both $\OF$ and $\OG$ are monotone, define the \emph{averaged forward-backward splitting iteration}
	\begin{equation}\label{eq:averagedfwdbackward}
		x^{k+1} = \frac{1}{2}x^{k} + \frac{1}{2}\OJ_{\alpha\OG}(x^{k} - \alpha \OF(x^{k})).
	\end{equation}
\end{algo}
\begin{theorem}[Convergence guarantees for forward-backward splitting]\label{thm:fwd-backwd}
	Let $\map{\OF}{\real^{n}}{\real^{n}}$ be Lipschitz \change{continuous} w.r.t.\ a diagonally weighted $\ell_{1}$ or $\ell_{\infty}$ norm $\|\cdot\|$, $\map{\OG}{\real^{n}}{\real^{n}}$ be continuous and monotone w.r.t.\ the same norm, and let $x^0 \in \real^{n}$ be arbitrary.
	\begin{enumerate}[(i)]
		\item\label{fwdbackwardOneMonotone} Suppose $\OF$ is strongly monotone w.r.t.\ $\|\cdot\|$ with monotonicity parameter $c > 0$, then the iteration~\eqref{eq:fwdbackward} converges to the unique zero, $x^{*}$, of $\OF+\OG$ for every $\alpha \in {]0, \frac{1}{\diagL(\OF)}]}$. Moreover, for every $k \in \mathbb{Z}_{\geq 0}$, the iteration satisfies
		\begin{equation*}
			\|x^{k+1} - x^{*}\| \leq (1 - \alpha c)\|x^{k}-x^{*}\|,
		\end{equation*}
		with the convergence rate being optimized at $\alpha = 1/\diagL(\OF)$.
		\item\label{fwdbackwardBothMonotone} If $\OF$ is monotone w.r.t.\ $\|\cdot\|$ and $\zero(\OF + \OG) \neq \emptyset$, then the iteration~\eqref{eq:averagedfwdbackward} converges to an element of $\zero(\OF + \OG)$ for every $\alpha \in {]0, \frac{1}{\diagL(\OF)}]}$.
	\end{enumerate}
\end{theorem}
\begin{proof}
	Regarding statement~(\ref{fwdbackwardOneMonotone}),~Lemmas~\ref{lemma:fwdstepLip}(\ref{1infnormfwdstep}) and~\ref{lemma:contractionResolvent} together imply that $\Lip(\OJ_{\alpha\OG} \circ \OS_{\alpha\OF}) \leq \Lip(\OJ_{\alpha\OG})\Lip(\OS_{\alpha\OF}) \leq 1-\alpha c < 1$ for all $\alpha \in {]0,1/\diagL(\OF)]}$. Then since $\fixed(\OJ_{\alpha\OF} \circ \OS_{\alpha\OF}) = \zero(\OF + \OG)$, the result is then a consequence of the Banach fixed point theorem.
	
	Statement~\ref{fwdbackwardBothMonotone} follows from $\Lip(\OJ_{\alpha\OG} \circ \OS_{\alpha\OF}) \leq 1$ and that the iteration~\eqref{eq:averagedfwdbackward} is the \KM iteration of the nonexpansive mapping $\OJ_{\alpha \OG} \circ \OS_{\alpha\OF}$ with $\theta = 1/2$.
\end{proof}

\begin{remark}[Comparison with standard forward-backward splitting convergence criteria]
	\begin{itemize}
		\item Compared to the Hilbert case, in the non-Euclidean setting, if both $\OF$ and $\OG$ are monotone, then iteration~\eqref{eq:averagedfwdbackward} must be applied to compute a zero of $\OF + \OG$. In the Hilbert case, iteration~\eqref{eq:fwdbackward} may be used instead since the composition of averaged operators is averaged. For non-Hilbert norms, the composition of two averaged operators need not be averaged.
		\item In monotone operator theory on Hilbert spaces, \change{Lipschitz continuity} of $\OF$ is not sufficient for the convergence of the iteration~\eqref{eq:fwdbackward}. Instead, \change{a standard sufficient condition for convergence is cocoercivity of $\OF$, see \citep{BM:80} and}~\citep[Theorem~26.14]{HHB-PLC:17}. In the case of diagonally-weighted $\ell_{1}/\ell_{\infty}$ norms, \change{Lipschitz continuity} is sufficient for convergence.
		This fact is due to the nonexpansiveness of $\OS_{\alpha\OF}$ for $\ell_{1}/\ell_{\infty}$ monotone $\OF$ and small enough $\alpha > 0$ as discussed in Remark~\ref{rmk:nonexpansive-fwdstep}.
	\end{itemize}
\end{remark}

\subsection{Peaceman-Rachford and Douglas-Rachford Splitting}

\begin{algo}[Peaceman-Rachford and Douglas-Rachford splitting]\label{alg:PR-DR}
	Let $\alpha > 0$. Then in~\citep[Section~7.3]{EKR-SB:16}, it is shown that
	\begin{equation}\label{eq:PRsplitting}
		(\OF + \OG)(x) = \vectorzeros[n] \quad \iff \quad (\OR_{\alpha\OF} \circ\OR_{\alpha\OG})(z) = z \text{ and } x = \OJ_{\alpha\OG} (z).
	\end{equation}
	The \emph{Peaceman-Rachford splitting method} corresponds to the fixed point iteration
	\begin{equation}\label{eq:PRiteration}
		\begin{aligned}
			x^{k+1} &= \OJ_{\alpha\OG}(z^{k}), \\
			z^{k+1} &= z^{k} + 2\OJ_{\alpha\OF}(2x^{k+1}-z^{k}) - 2x^{k+1}.
		\end{aligned}
	\end{equation}
	If both $\OF$ and $\OG$ are monotone, the term $\OR_{\alpha\OF} \circ \OR_{\alpha\OG}$ in~\eqref{eq:PRsplitting} is averaged to yield
	\begin{equation}\label{eq:DRsplitting}
		(\OF + \OG)(x) = \vectorzeros[n] \quad \iff \quad \Big(\frac{1}{2}\Id + \frac{1}{2}\OR_{\alpha\OF} \circ \OR_{\alpha\OG}\Big)(z) = z \text{ and } x = \OJ_{\alpha\OG}(z).
	\end{equation}
	The fixed point iteration corresponding to~\eqref{eq:DRsplitting} is called the \emph{Douglas-Rachford splitting method} and is given by
	\begin{equation}\label{eq:DRiteration}
		\begin{aligned}
			x^{k+1} &= \OJ_{\alpha\OG}(z^{k}), \\
			z^{k+1} &= z^{k} + \OJ_{\alpha\OF}(2x^{k+1} - z^{k}) - x^{k+1}.
		\end{aligned}
	\end{equation}
\end{algo}

\begin{theorem}[Convergence guarantees for Peaceman-Rachford and Douglas-Rachford splitting]\label{thm:PR-DR}
	Let both $\map{\OF}{\real^{n}}{\real^{n}}$ and $\map{\OG}{\real^{n}}{\real^{n}}$ be Lipschitz \change{continuous} w.r.t.\ a diagonally weighted $\ell_{1}$ or $\ell_{\infty}$ norm $\|\cdot\|$, let $\OG$ be monotone w.r.t.\ the same norm, and let $x^0 \in \real^{n}$.
	\begin{enumerate}[(i)]
		\item\label{PR} Suppose $\OF$ is strongly monotone w.r.t.\ $\|\cdot\|$ with monotonicity parameter $c > 0$. Then the sequence of $\{x_k\}_{k=0}^\infty$ generated by the iteration~\eqref{eq:PRiteration} converges to the unique zero, $x^{*}$, of $\OF+\OG$ for every $\alpha \in {\Big]0, \min\Big\{\frac{1}{\diagL(\OF)},\frac{1}{\diagL(\OG)}\Big\}\Big]}$. Moreover, for every $k \in \mathbb{Z}_{\geq 0}$, the iteration satisfies
		\begin{equation*}
			\|x^{k+1} - x^{*}\| \leq \frac{1-\alpha c}{1 + \alpha c}\|x^{k}-x^{*}\|,
		\end{equation*}
		with the convergence rate being optimized at $\alpha = \min\Big\{\frac{1}{\diagL(\OF)},\frac{1}{\diagL(\OG)}\Big\}$.
		\item\label{DR} Alternatively suppose $\OF$ is monotone w.r.t.\ $\|\cdot\|$ and $\zero(\OF + \OG) \neq \emptyset$. Then the sequence of $\{x_k\}_{k=0}^\infty$ generated by the iteration~\eqref{eq:DRiteration} converges to an element of $\zero(\OF + \OG)$ for every $\alpha \in {\Big]0, \min\Big\{\frac{1}{\diagL(\OF)},\frac{1}{\diagL(\OG)}\Big\}\Big]}$.
	\end{enumerate}
\end{theorem}

\begin{proof}
	Regarding statement~(\ref{PR}), by Lemma~\ref{lemma:CayleyLip}(\ref{LipCayley1inf}), we have that $$\Lip(\OR_{\alpha\OF} \circ \OR_{\alpha\OG}) \leq \Lip(\OR_{\alpha\OF})\Lip(\OR_{\alpha\OG}) \leq \frac{1-\alpha c}{1 + \alpha c} < 1$$ for $\alpha \in {]0, \min\{1/\diagL(\OF),1/\diagL(\OG)\}]}$. Then since $\Lip(\OJ_{\alpha \OG})$ is nonexpansive, the Banach fixed point theorem implies the result.
	
	Statement~(\ref{DR}) holds because Lemma~\ref{lemma:CayleyLip}(\ref{LipCayley1inf}) implies $\Lip(\OR_{\alpha\OF} \circ \OR_{\alpha\OG}) \leq 1$. Then the iteration~\eqref{eq:DRiteration} converges because of Lemma~\ref{lemma:KMconvergence}.
\end{proof}

Compared to classical criteria for the convergence of the Douglas-Rachford iteration, Theorem~\ref{thm:PR-DR} requires \change{Lipschitz continuity} of $\OF$ and $\OG$ in order to utilize the Lipschitz estimates for the reflected resolvents $\OR_{\alpha\OF}$ and $\OR_{\alpha\OG}$. Moreover, the parameter $\alpha > 0$ must be chosen small enough in the non-Euclidean setting whereas convergence is guaranteed for any choice of $\alpha$ in the Hilbert case. This is because the reflected resolvent is only nonexpansive for a certain range of $\alpha$ when the norm is not a Hilbert one, see Example~\ref{example:linear}.

\newcommand{\ReLU}{\mathrm{ReLU}}
\newcommand{\LeakyReLU}{\mathrm{LeakyReLU}}
\section{Set-Valued Inclusions and an Application to Recurrent Neural Networks}\label{sec:applications}

\subsection{Set-Valued Inclusions and Non-Euclidean Properties of Proximal Operators}\label{subsec:set-valued}

In many instances one may wish to solve an inclusion problem of the form
\begin{equation}\label{eq:mon-splitting}
	\text{Find } x \in \real^{n} \text{ such that } \quad \vectorzeros[n] \in (\OF + \OG)(x),
\end{equation}
where $\map{\OF}{\real^{n}}{\real^{n}}$ is a single-valued continuous monotone mapping but $\map{\OG}{\real^{n}}{2^{\real^{n}}}$ is a set-valued mapping. In monotone operator theory on Hilbert spaces, leveraging the fact that $\OJ_{\alpha \OG}$ is single-valued and nonexpansive for every $\alpha > 0$ when $\OG$ is maximally monotone, algorithms such as the forward-backward splitting and Douglas-Rachford splitting may be used to solve~\eqref{eq:mon-splitting} under suitable assumptions on $\OF$. 

In this section we aim to prove similar results in the non-Euclidean case. We will specialize to the case that $\OG$ is the subdifferential of a \emph{separable, proper lower semicontinuous (l.s.c.), convex function}. To start we must define the proximal operator of a l.s.c. convex function.

\begin{definition}[Proximal operator,~\change{\citealt{JJM:62} and}~{\citealt[Def.~12.23]{HHB-PLC:17}}]
	Let $\map{g}{\real^{n}}{{]{-}\infty,{+}\infty]}}$ be a proper l.s.c. convex function. The \emph{proximal operator of $g$} evaluated at $x \in \real^{n}$ is the map $\map{\prox_g}{\real^{n}}{\real^{n}}$ defined by
	\begin{equation}\label{eq:prox}
		\prox_g(x) = \argmin_{z \in \real^{n}} \frac{1}{2}\|x - z\|_{2}^2 + g(z).
	\end{equation}
\end{definition}

Since $\map{g}{\real^{n}}{{]{-}\infty,{+}\infty]}}$ is proper, l.s.c., and convex, we can see that for $\alpha > 0$ and fixed $x \in \real^{n}$, the map $z \mapsto \frac{1}{2}\|x - z\|_{2}^2 + \alpha g(z)$ is strongly convex and thus has a unique minimizer, so for each $x \in \real^{n}$, $\prox_{\alpha g}(x)$ is single-valued. 
Moreover, we have the following connection between proximal operators and resolvents of subdifferentials. 
\begin{proposition}[\change{\citealt{RTR:76b} and}~{\citealt[Example~23.3]{HHB-PLC:17}}]\label{prop:resolvent-proximal}
	Suppose $\map{g}{\real^{n}}{{]{-}\infty,{+}\infty]}}$ is proper, l.s.c., and convex. Then for every $\alpha > 0$,
	$\OJ_{\alpha \partial g}(x) = \prox_{\alpha g}(x)$.
\end{proposition}

In the case of scalar functions, one can exactly capture the set of functions which are proximal operators of some proper l.s.c. convex functions.

\begin{proposition}[{\citealt[Proposition~24.31]{HHB-PLC:17}}]\label{prop:prox-nec-suff}
	Let $\map{\phi}{\real}{\real}$. Then $\phi$ is the proximal operator of a proper l.s.c. convex function $\map{g}{\real}{{]{-}\infty,{+}\infty]}}$, i.e., $\phi = \prox_g$ if and only if $\phi$ satisfies
	\begin{equation}\label{eq:slope-restriction}
		0 \leq \frac{\phi(x) - \phi(y)}{x - y} \leq 1, \quad \text{for all } x,y \in \real, x\neq y.
	\end{equation}
\end{proposition}

A list of examples of scalar functions satisfying~\eqref{eq:slope-restriction} and their corresponding proper l.s.c. convex function is provided in~\citep[Table~1]{JL-CF-ZL:19}.

To prove non-Euclidean properties of proximal operators, we will leverage a well-known property, which we highlight in the following proposition.

\begin{proposition}[Proximal operator of separable convex functions,~{\citealt[Section~2.1]{NP-SB:14}}]\label{prop:prox-properties}
	For $i \in \until{n}$, let $\map{g_i}{\real}{{]{-}\infty,{+}\infty]}}$, be proper, l.s.c., and convex. Define $\map{g}{\real^{n}}{{]{-}\infty,{+}\infty]}}$ by $g(x) = \sum_{i=1}^n g_i(x_i)$. Then $g$ is proper, l.s.c., and convex and for all $\alpha > 0$,
	\begin{equation*}
		\prox_{\alpha g}(x) = (\prox_{\alpha g_{1}}(x_{1}), \dots, \prox_{\alpha g_n}(x_n)) \in \real^{n}.
	\end{equation*}
	If $g$ satisfies $g(x) = \sum_{i=1}^n g_i(x_i)$ with each $g_i$ proper, l.s.c., and convex, we call $g$ \emph{separable}.
\end{proposition}

In the following novel proposition, we showcase that when $g$ is separable, $\prox_{\alpha g}$ and $2\prox_{\alpha g} - \Id$ are nonexpansive w.r.t.\ non-Euclidean norms.


\begin{proposition}[Nonexpansiveness of proximal operators of separable convex maps]\label{prop:nonexpansive}
	For $i \in \until{n}$, let each $\map{g_i}{\real}{{]{-}\infty,{+}\infty]}}$ be proper, l.s.c., and convex. Define $\map{g}{\real^{n}}{{]{-}\infty,{+}\infty]}}$ by $g(x) = \sum_{i=1}^n g_i(x_i)$. For every $\alpha > 0$ and for any $\eta \in \realpositive^n$, both $\OJ_{\alpha \partial g} = \prox_{\alpha g}$ and $\OR_{\alpha\partial g} = 2\prox_{\alpha g} - \Id$ are nonexpansive w.r.t.\ $\|\cdot\|_{\infty,[\eta]^{-1}}$.\footnote{More generally, $\prox_{\alpha g}$ and $2\prox_{\alpha g} - \Id$ are nonexpansive with respect to any monotonic norm.}
\end{proposition}

\begin{proof}
	By Proposition~\ref{prop:prox-properties} we have
	$\prox_{\alpha g}(x) = (\prox_{\alpha g_{1}}(x_{1}), \dots, \prox_{\alpha g_n}(x_n))$.
	Moreover, each $\prox_{\alpha g_i}$ is nonexpansive and monotone by Proposition~\ref{prop:prox-nec-suff} and thus satisfies
	\begin{equation}\label{eq:prox-mon}
		0 \leq (\prox_{\alpha g_i}(x_i) - \prox_{\alpha g_i}(y_i))(x_i-y_i) \leq (x_i-y_i)^2, \qquad \text{for all } x_i,y_i \in \real.	\end{equation}
	
	We then conclude
	\begin{equation*}
		\begin{aligned}
			\|\prox_{\alpha g}(x) - \prox_{\alpha g}(y)\|_{\infty,[\eta]^{-1}} &= \max_{i\in\until{n}} \frac{1}{\eta_i}|\prox_{\alpha g_i}(x_i) - \prox_{\alpha g_i}(y_i)| \\
			&\leq \max_{i\in\until{n}} \frac{1}{\eta_i} |x_i - y_i| = \|x - y\|_{\infty,[\eta]^{-1}}.
		\end{aligned}
	\end{equation*}
	Regarding $\OR_{\alpha\partial g}$, we note that~\eqref{eq:prox-mon} implies for all $x_i,y_i \in \real$
	\begin{equation*}
		\begin{aligned}
			-(x_i-y_i)^2 \leq &\;((2\prox_{\alpha g_i}(x_i) - x_i) - (2\prox_{\alpha g_i}(y_i) - y_i))(x_i-y_i) \leq (x_i-y_i)^2,  \\
			\implies &|(2\prox_{\alpha g_i}(x_i) - x_i) - (2\prox_{\alpha g_i}(y_i) - y_i)| \leq |x_i-y_i|.
		\end{aligned}
	\end{equation*}
	Following the same reasoning as for $\prox_{\alpha g}$, we conclude that $2\prox_{\alpha g} - \Id$ is nonexpansive w.r.t.\ $\|\cdot\|_{\infty,[\eta]^{-1}}$.
\end{proof}

We recall from monotone operator theory on Hilbert spaces that if $\map{\OF}{\real^{n}}{\real^{n}}$ is continuous and $\map{\OG}{\real^{n}}{2^{\real^{n}}}$ satisfies $\dom(\OJ_{\alpha \OG}) = \real^{n}$ and $\OJ_{\alpha \OG}(x)$ is single-valued for all $x \in \real^{n}, \alpha > 0$, then the following equivalences hold: (i) $\vectorzeros[n] \in (\OF + \OG)(x)$, (ii) $x = (\OJ_{\alpha \OG} \circ \OS_{\alpha\OF})(x)$, and (iii) $z = (\OR_{\alpha\OF} \circ \OR_{\alpha\OG})(z)$ and $x = \OJ_{\alpha \OG}(z)$~\citep[pp.~25 and~28]{EKR-SB:16}. In other words, even if $\OG$ is a set-valued mapping, forward-backward and Peaceman-Rachford splitting methods may be applied to compute zeros of the inclusion problem~\eqref{eq:mon-splitting}.

When $\map{\OF}{\real^{n}}{\real^{n}}$ in~\eqref{eq:mon-splitting} is Lipschitz \change{continuous} and strongly monotone w.r.t.\ $\|\cdot\|_{\infty,[\eta]^{-1}}$ with monotonicity parameter $c > 0$ and $\OG = \partial g$ for a separable proper, l.s.c., convex mapping $\map{g}{\real^{n}}{{]{-}\infty,{+}\infty]}}$, by Proposition~\ref{prop:nonexpansive}, the composition $\prox_{\alpha g} \circ \OS_{\alpha\OF}$ is a contraction w.r.t.\ $\|\cdot\|_{\infty,[\eta]^{-1}}$ for small enough $\alpha > 0$. Therefore, the forward-backward splitting method, Algorithm~\ref{alg:fwd-backwd}, may be applied to find a zero of the splitting problem~\eqref{eq:mon-splitting}. Analogously, for small enough $\alpha > 0$, $\OR_{\alpha\OF} \circ \OR_{\alpha\OG}$ is a contraction w.r.t.\ $\|\cdot\|_{\infty,[\eta]^{-1}}$ and Peaceman-Rachford splitting, Algorithm~\ref{alg:PR-DR}, may be applied to find a zero of the problem~\eqref{eq:mon-splitting}. In the following section, we present an application of the above theory to recurrent neural networks.

\subsection{Iterations for Recurrent Neural Network Equilibrium Computation}\label{subsec:RNN-theory}
Consider the continuous-time recurrent neural network
\begin{equation}\label{eq:inn}
	\begin{aligned}
		\dot{x} = -x + \Phi(Ax + Bu + b) =: F(x,u),\\
	\end{aligned}
\end{equation}
where $x \in \real^{n}, u \in \real^{m}, A \in \real^{n \times n}, B \in \real^{n \times m}, b \in \real^{n}$, and $\map{\Phi}{\real^{n}}{\real^{n}}$ is a separable activation function, i.e., it acts entry-wise in the sense that $\Phi(x) = (\phi(x_{1}),\dots,\phi(x_n))^\top$. In this section we consider activation functions $\map{\phi}{\real}{\real}$ satisfying slope bounds of the form
\begin{equation}\label{eq:slope-ineq}
	d_{1} = \inf_{x,y \in \real, x\neq y} \frac{\phi(x)-\phi(y)}{x-y} \geq 0, \qquad d_{2} = \sup_{x,y \in \real, x\neq y} \frac{\phi(x)-\phi(y)}{x-y} \leq 1.
\end{equation}

Most standard activation functions used in machine learning satisfy these bounds. 
In~\citep[Theorem~23]{AD-AVP-FB:22q}, it was shown that a sufficient condition for the strong infinitesimal contractivity of the map $x \mapsto F(x,u)$ is the existence of weights $\eta \in \realpositive^n$ such that $\mu_{\infty,[\eta]^{-1}}(A) < 1$; if this condition holds, the recurrent neural network~\eqref{eq:inn} is strongly infinitesimally contracting w.r.t.\ $\|\cdot\|_{\infty,[\eta]^{-1}}$ with rate $1 - \max\{d_{1}\gamma, d_{2}\gamma\}$, where we define $\gamma = \mu_{\infty,[\eta]^{-1}}(A) < 1$.

Suppose that, for fixed $u$, we are interested in efficiently computing the unique equilibrium point $x_{u}^{*}$ of $F(x,u)$. Note that equilibrium points $x_{u}^{*}$ satisfy $x_{u}^{*} = \Phi(Ax_{u}^{*} + Bu + b)$ which corresponds to an implicit neural network (INN), which have recently gained significant attention in the machine learning community~\citep{SB-JZK-VK:19,EW-JZK:20,LEG-FG-BT-AA-AYT:21}. In this regard, computing equilibrium points of~\eqref{eq:inn} corresponds to computing the forward pass of an INN.

Since the map $x \mapsto F(x,u)$ is strongly infinitesimally contracting w.r.t.\ $\|\cdot\|_{\infty,[\eta]^{-1}}$, the map $x \mapsto -F(x,u)$ is strongly monotone with monotonicity parameter $1 - \max\{d_{1}\gamma,d_{2}\gamma\}$ (see Remark~\ref{rmk:contraction}). As a consequence, applying the forward step method, Algorithm~\ref{alg:fwdstep}, to compute $x_{u}^{*}$ yields the iteration
\begin{equation}\label{eq:averagediteration}
	x^{k+1} = (1 - \alpha)x^{k} + \alpha \Phi(Ax^{k} + Bu + b),
\end{equation}
which is the iteration proposed in~\citep{SJ-AD-AVP-FB:21f}. This iteration is guaranteed to converge for every $\alpha \in {]0,\frac{1}{1-\min_{i\in \until{n}}\min\{d_{1} \cdot (A)_{ii},d_{2}\cdot(A)_{ii}\}}]}$ with contraction factor $1 - \alpha(1 - \max\{d_{1}\gamma,d_{2}\gamma\})$ by Theorem~\ref{thm:fwdstep}(ii).

However, rather than viewing finding an equilibrium of~\eqref{eq:inn} as finding a zero of a non-Euclidean monotone operator, it is also possible to view it as a monotone inclusion problem of the form~\eqref{eq:mon-splitting}. 

\begin{proposition}[{\citealt[Theorem~1]{EW-JZK:20}}]\label{prop:splittingRNN}
	Suppose $\phi$ satisfies the bounds~\eqref{eq:slope-ineq}. Then finding an equilibrium point $x_{u}^{*}$ of~\eqref{eq:inn} is equivalent to the (set-valued) operator splitting problem $\vectorzeros[n] \in (\OF+\OG)(x_{u}^{*})$, with
	\begin{equation}\label{eq:splitting}
		\OF(z) = (I_n - A)z - (Bu + b), \qquad \OG(z) = \partial g(z),
	\end{equation}
	where we denote $g(z) = \sum_{i=1}^n f(z_i)$ and $\map{f}{\real}{{]{-}\infty,{+}\infty]}}$ is proper, l.s.c., convex, and satisfies $\phi = \prox_{f}$.
\end{proposition}
\begin{proof}
	By Proposition~\ref{prop:prox-nec-suff}, since $\phi$ satisfies the bounds~\eqref{eq:slope-ineq}, there exists a proper, l.s.c., convex $f$ with $\phi = \prox_f$. The remainder of the proof is equivalent to that in~\citep[Thm~1]{EW-JZK:20}.
\end{proof}

While Proposition~\ref{prop:splittingRNN} was leveraged in~\citep{EW-JZK:20} for monotonicity w.r.t.\ the $\ell_{2}$ norm, we will use it for $\OF$ which is monotone w.r.t.\ a diagonally-weighted $\ell_{\infty}$ norm.$\;\!$\footnote{Unless $A = A^\top$, the monotone inclusion problem~\eqref{eq:splitting} does not arise from a convex minimization problem.}


Checking that $\OF$ is strongly monotone w.r.t.\ $\|\cdot\|_{\infty,[\eta]^{-1}}$ is straightforward under the assumption that $\gamma < 1$.
As a consequence of Propositions,~\ref{prop:nonexpansive} and~\ref{prop:splittingRNN}, we can consider different operator splitting algorithms to compute the equilibrium of~\eqref{eq:inn}. First, the forward-backward splitting method, Algorithm~\ref{alg:fwd-backwd}, as applied to this problem is
\begin{equation}\label{eq:FB-RNN}
	\begin{aligned}
		x^{k+1} &= \prox_{\alpha g}((1 - \alpha)x^{k} + \alpha (Ax^{k} + Bu + b)).
	\end{aligned}
\end{equation}
Since $\OF$ is Lipschitz \change{continuous}, this iteration is guaranteed to converge to the unique fixed point of~\eqref{eq:inn} by Theorem~\ref{thm:fwd-backwd}(i). Moreover, the contraction factor for this iteration is $1 - \alpha(1 - \gamma)$ for $\alpha \in {]0, \frac{1}{1 - \min_i (A)_{ii}}]}$, with contraction factor being minimized at $\alpha^{*} = \frac{1}{1 - \min_i (A)_{ii}}$. Note that compared to the iteration~\eqref{eq:averagediteration}, iteration~\eqref{eq:FB-RNN} has a larger allowable range of step sizes and improved contraction factor at the expense of computing a proximal operator at each iteration.

Alternatively, the fixed point may be computed by means of the Peaceman-Rachford splitting method, Algorithm~\ref{alg:PR-DR}, which can be written
\begin{equation}\label{eq:PR-RNN}
	\begin{aligned}
		x^{k+1} &= (I_n + \alpha(I_n - A))^{-1}(z^{k} + \alpha(Bu + b)), \\
		z^{k+1} &= z^{k} + 2\prox_{\alpha g}(2x^{k+1}-z^{k}) - 2x^{k+1}.
	\end{aligned}
\end{equation}
Since $\OF$ is Lipschitz \change{continuous} and $\OR_{\alpha\OG}$ is nonexpansive for every $\alpha > 0$, this iteration converges to the unique fixed point of~\eqref{eq:inn} for $\alpha$ in a suitable range by Theorem~\ref{thm:PR-DR}(i). Moreover, the contraction factor is $\frac{1 - \alpha(1 - \gamma)}{1 + \alpha(1 - \gamma)}$ for
$\alpha \in {]0, \frac{1}{1 - \min_i (A)_{ii}}]},$
which comes from the Lipschitz constant of $\OF$. In other words, the contraction factor is improved for Peaceman-Rachford compared to forward-backward splitting and the range of allowable step sizes is identical. For RNNs where $(I_n + \alpha(I_n - A))$ may be easily inverted, this splitting method may be preferred.

\subsection{Numerical Implementations}
\begin{figure}[h]
	\centering
	\begin{tabular}{ccc}
		&$\gamma = 0.9$ & $\gamma = -1$ \\
		\raisebox{5.2em}{$\LeakyReLU$} &\includegraphics[width=0.39\linewidth]{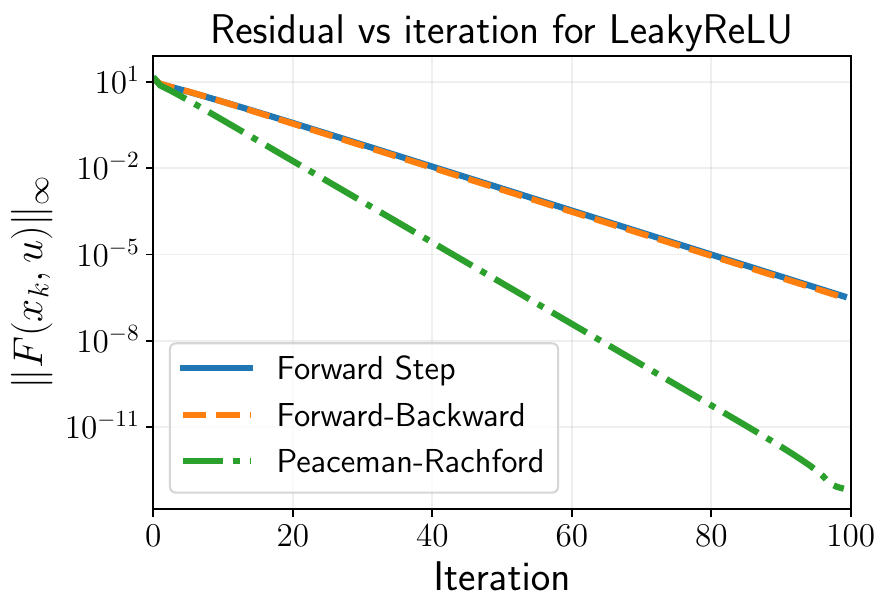} & \includegraphics[width=0.39\linewidth]{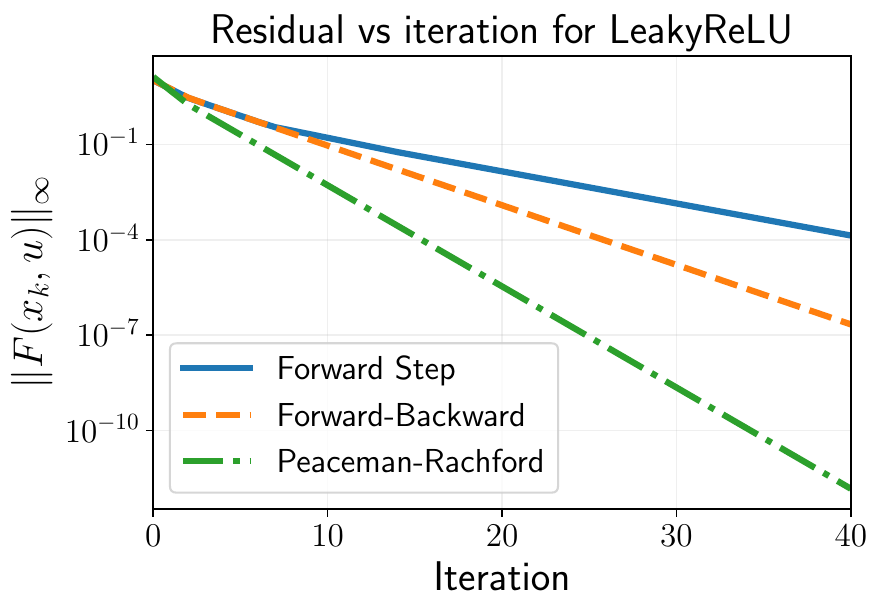} \\
		\raisebox{5.2em}{$\ReLU$} &\includegraphics[width=0.39\linewidth]{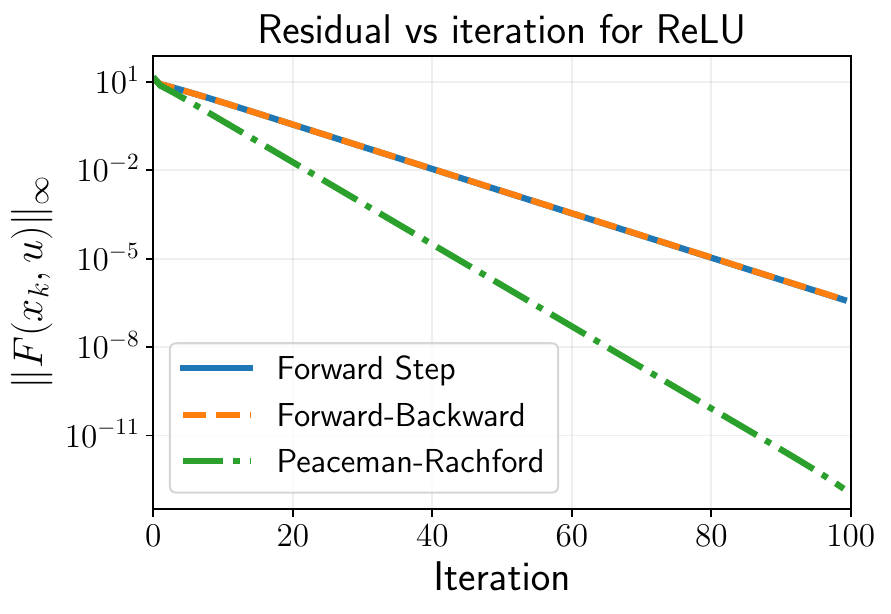} & \includegraphics[width=0.39\linewidth]{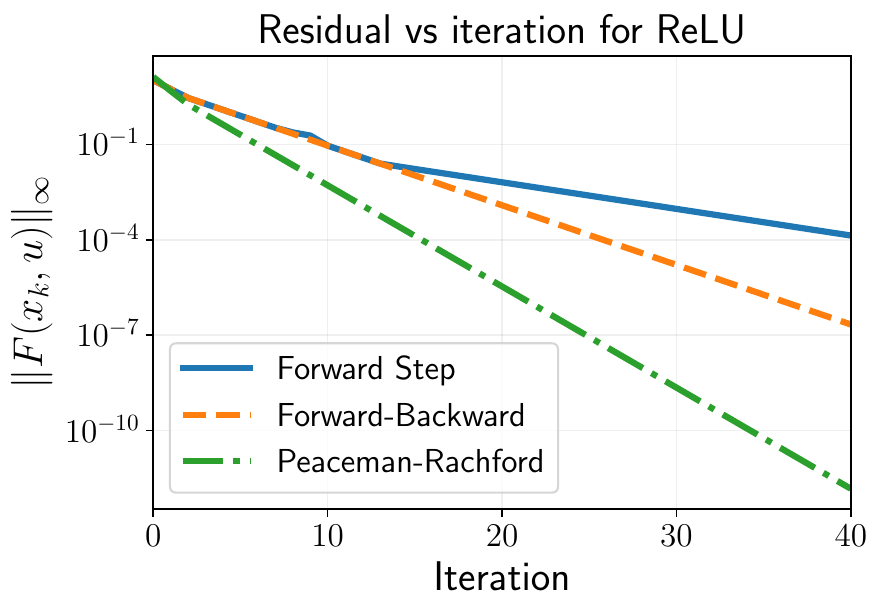}
	\end{tabular}
	\caption{Residual versus number of iterations for forward-step method~\eqref{eq:averagediteration}, forward-backward splitting~\eqref{eq:FB-RNN}, and Peaceman-Rachford splitting~\eqref{eq:PR-RNN} for computing the equilibrium of the recurrent neural network~\eqref{eq:inn}. The top two plots correspond to $\phi = \LeakyReLU$ with $a = 0.1$ and the bottom two plots correspond to $\phi = \ReLU$. The left two plots correspond to $\gamma = 0.9$ and the right two plots correspond to $\gamma = -1$. Curves for the forward-step method and forward-backward splitting are directly on top of one another in the left two plots. Note the difference in the number of iterations with respect to the parameter $\gamma$.}\label{fig:iterations}
\end{figure}
To assess the efficacy of the iterations in~\eqref{eq:averagediteration},~\eqref{eq:FB-RNN}, and~\eqref{eq:PR-RNN}, we generated $A, B, b, u$ in~\eqref{eq:inn} and applied the iterations to compute the equilibrium. We generate $A \in \real^{200 \times 200}, B \in \real^{200 \times 50}, u \in \real^{50}, b \in \real^{200}$ with entries normally distributed as $A_{ij}, B_{ij}, b_{i}, u_i \sim \mathcal{N}(0, 1)$. To ensure that $A \in \real^{200 \times 200}$ satisfies the constraint $\mu_{\infty,[\eta]^{-1}}(A) \leq \gamma$ for some $\eta \in \realpositive^{200}$, we pick $[\eta] = I_{200}$ and orthogonally project $A$ onto the convex polytope $\setdef{A \in \real^{200 \times 200}}{\mu_{\infty}(A) \leq \gamma}$ using CVXPY~\citep{SD-SB:16}. In experiments, we consider $\gamma \in \{-1, 0.9\}$ and consider activation functions $\phi(x) = \ReLU(x) = \max\{x,0\}$ and $\phi(x) = \LeakyReLU(x) = \max\{x,ax\}$ with $a = 0.1$.\footnote{Note that the slope bounds from~\eqref{eq:slope-restriction} are $d_{1} = 0, d_{2} = 1$ for $\ReLU$ and $d_{1} = a, d_{2} = 1$ for $\LeakyReLU$ with $a \in {[0,1[}.$} The proper, l.s.c., convex $f$ corresponding to these activation functions are available in~\citep[Table~1]{JL-CF-ZL:19}.

For all iterations, we initialize $x^0$ at the origin and for the Peaceman-Rachford iteration, we initialize $z^0$ at the origin. For each iteration we pick the largest theoretically allowable step size, which in all cases was $\frac{1}{1-\min_i (A)_{ii}}$ (since $\min_{i\in \until{n}} (A)_{ii}$ was negative in all cases). For the case of $\gamma = 0.9$, we found that the largest theoretically allowable step size was $\alpha \approx 0.182$ and for $\gamma = -1$ the largest step size was $\alpha \approx 0.175$. The plots of the residual $\|x_k - \Phi(Ax_k + Bu + b)\|_{\infty} = \|F(x_k,u)\|_{\infty}$ versus the number of iterations for all different cases is shown in Figure~\ref{fig:iterations}.$\;\!$\footnote{All iterations and graphics were run and generated in Python. Code to reproduce experiments is available at \url{https://github.com/davydovalexander/RNN-Equilibrium-NonEucMonotone}.}

We see that, when $\gamma = 0.9$, both forward-step and forward-backward splitting methods for computing the equilibrium of~\eqref{eq:inn} converge at the same rate. This result agrees with the theory since $\gamma > 0$, so that $\max\{d_{1}\gamma,d_{2}\gamma\} = \gamma$ for both $\ReLU$ and $\LeakyReLU$ and the estimated contraction factor for both the forward step method and forward-backward splitting is $1 - \alpha(1 - \gamma) \approx 0.982$. For the Peaceman-Rachford splitting method and $\gamma = 0.9$, the estimated contraction factor is $\frac{1-\alpha(1-\gamma)}{1+\alpha(1-\gamma)} \approx 0.964$, which justifies the improved rate of convergence. When $\gamma = -1$, the forward-backward splitting method converges faster than the forward step method. This result agrees with the theory since the estimated contraction factor for the forward step method is $1 - \alpha(1-\phi(\gamma)) \approx 0.807$ in the case of $\LeakyReLU$ and $\approx 0.825$ in the case of $\ReLU$ while the estimated contraction factor for forward-backward splitting is $1 - \alpha(1-\gamma) \approx 0.649$ independent of activation function. On the other hand, for the Peaceman-Rachford splitting method and $\gamma = -1$, the estimated contraction factor is $\frac{1-\alpha(1-\gamma)}{1+\alpha(1-\gamma)} \approx 0.481$, which justifies the improved rate of convergence.

\subsection{Tightened Lipschitz Constants for Continuous-Time RNNs}
We are interested in studying the robustness of the RNN~\eqref{eq:inn} to input perturbations. In other words, given a nominal input, $u$, and its corresponding equilibrium output, $x_{u}^{*}$, we aim to upper-bound the deviation of the output due to a change in the input. The Lipschitz constant of a neural network is one common metric used to evaluate its robustness, as discussed in works such as~\citep{MF-AR-HH-MM-GJP:19,PLC-JCP:20,PP-AK-JB-PK-FA:21}. In the context of implicit neural networks, Lipschitz constants have been studied in~\citep{MR-RW-IRM:20,CP-EW-JZK:21,SJ-AD-AVP-FB:21f}, with~\citep{CP-EW-JZK:21} unrolling forward-backward splitting iterations to provide $\ell_{2}$ Lipschitz estimates. In what follows, we generalize the procedure in~\citep{CP-EW-JZK:21} using techniques from non-Euclidean monotone operator theory to provide novel and tighter $\ell_{\infty}$ Lipschitz estimates.

\begin{theorem}[Lipschitz estimate of equilibrium points of~\eqref{eq:inn}]\label{thm:RNN-Lip}
	Suppose that $A$ satisfies $\mu_{\infty,[\eta]^{-1}}(A) = \gamma < 1$ for some $\eta \in \realpositive^n$ and that $\phi = \prox_f$ for some proper, l.s.c., convex $\map{f}{\real}{{]{-}\infty,+{\infty}}]}$. Define $\map{f_{\ON}}{\real^{m}}{\real^{n}}$ by $f_{\ON}(u) = x_{u}^{*}$ where $x_{u}^{*}$ solves the fixed point problem $x_{u}^{*} = \Phi(Ax_{u}^{*} + Bu + b).$\footnote{Note that if $x_{u}^{*}$ solves the fixed point problem $x_{u}^{*} = \Phi(Ax_{u}^{*} + Bu + b)$, then it is an equilibrium point of the RNN~\eqref{eq:inn}.} Then for $\eta_{\max} = \max_{i\in\until{n}} \eta_i$, $\eta_{\min} = \min_{i\in \until{n}} \eta_i$, and $\Lip_{\infty}(f_\ON)$ denoting the minimal $\ell_{\infty}$ Lipschitz constant of $f_\ON$,
	\begin{equation}
		\Lip_{\infty}(f_{\ON}) \leq \frac{\eta_{\max}}{\eta_{\min}}\frac{\|B\|_{\infty}}{1 - \gamma}.
	\end{equation}
\end{theorem}

\begin{proof}
	We consider the forward-backward splitting iteration given input $u$ as $x_{u}^{k+1} = \prox_{\alpha g}((1-\alpha)x_{u}^{k} + \alpha(Ax_{u}^{k} + Bu + b))$ with initial condition $x_{u}^0 = \vectorzeros[n]$ which is guaranteed to converge for $\alpha \in {]0, \frac{1}{1 - \min_i (A)_{ii}}]}$ since $\prox_{\alpha g}$ is nonexpansive and $\OS_{\alpha\OF}$ is a contraction w.r.t.\ $\|\cdot\|_{\infty,[\eta]^{-1}}$ for every $\alpha$ in this range where $\OF$ is defined as in~\eqref{eq:splitting}. We find
	\begin{align}
		\|x_{u}^{k} - x_{v}^{k}\|_{\infty,[\eta]^{-1}} &= \|\prox_{\alpha g}((1-\alpha)x_{u}^{k-1} + \alpha(Ax_{u}^{k-1} + Bu + b)) \nonumber\\ & \quad - \prox_{\alpha g}((1-\alpha)x_{v}^{k-1} + \alpha(Ax_{v}^{k-1} + Bv + b))\|_{\infty,[\eta]^{-1}} \nonumber\\
		&\leq \|\OS_{\alpha(\Id-A)}(x_{u}^{k-1} - x_{v}^{k-1})\|_{\infty,[\eta]^{-1}} + \alpha\|B(u-v)\|_{\infty,[\eta]^{-1}}\label{eq:prox-nonex} \\
		&\leq \Lip(\OS_{\alpha(\Id-A)})^{k}\|x_{u}^{0} - x_{v}^{0}\|_{\infty,[\eta]^{-1}} + \alpha\|B(u-v)\|_{\infty,[\eta]^{-1}} \sum_{i=0}^{k-1} \Lip(\OS_{\alpha(\Id-A)})^i \nonumber\\
		&= \alpha\|B(u-v)\|_{\infty,[\eta]^{-1}}\sum_{i=0}^{k-1} \Lip(\OS_{\alpha(\Id-A)})^i \label{eq:initial-cond},
	\end{align}
	where~\eqref{eq:prox-nonex} holds because of nonexpansiveness of $\prox_{\alpha g}$ and the triangle inequality
	and~\eqref{eq:initial-cond} is a consequence of $x_{u}^0 = x_{v}^0 = \vectorzeros[n]$.
	
	Since the forward-backward splitting iteration converges for every $\alpha$ in the desired range, we can take the limit as $k \to \infty$ and find that $x_{u}^{k} \to x_{u}^{*}$ and $x_{v}^{k} \to x_{v}^{*}$ as $k \to \infty$. Then
	\begin{align}
		\|x_{u}^{*} - x_{v}^{*}\|_{\infty,[\eta]^{-1}} &\leq \alpha\|B(u-v)\|_{\infty,[\eta]^{-1}}\sum_{i=0}^{\infty} \Lip(\OS_{\alpha(\Id-A)})^i \\
		&= \frac{\alpha\|B(u-v)\|_{\infty,[\eta]^{-1}}}{1 - \Lip(\OS_{\alpha(\Id-A)})} \leq \frac{\alpha\|B(u-v)\|_{\infty,[\eta]^{-1}}}{1 - (1- \alpha (1-\gamma))} = \frac{\|B(u-v)\|_{\infty,[\eta]^{-1}}}{1-\gamma},
	\end{align}
	which implies the result because $\eta_{\max}^{-1}\|z\|_{\infty} \leq \|z\|_{\infty,[\eta]^{-1}} \leq \eta_{\min}^{-1}\|z\|_{\infty}$ for every $z \in \real^{n}$.
\end{proof}

\begin{remark}
	In~\citep[Corollary~5]{SJ-AD-AVP-FB:21f}, the following Lipschitz estimate is given:
	\begin{equation}\label{eq:old-Lip-est}
		\Lip_{\infty}(f_{\ON}) \leq \frac{\eta_{\max}}{\eta_{\min}}\frac{\|B\|_{\infty}}{1 - \max\{\gamma,0\}}.
	\end{equation}
	The Lipschitz estimate in Theorem~\ref{thm:RNN-Lip} is always a tighter bound than the estimate~\eqref{eq:old-Lip-est} and allows the choice of negative $\gamma$ to further lower the Lipschitz constant of the RNN. Indeed, one way to make the neural network more robust to uncertainties in its input would be to ensure that $\gamma$ is a large negative number.
\end{remark}

\section{Conclusion}
In this paper, we introduce a non-Euclidean version of classical results in monotone operator theory with a focus on mappings that are monotone with respect to diagonally-weighted $\ell_{1}$ or $\ell_{\infty}$ norms. Our results show that the resolvent and reflected resolvent maintain many useful properties from the Hilbert case, and we prove that commonly used algorithms for finding zeros of monotone operators and their sums remain effective in the non-Euclidean setting. We applied our theory to the problem of equilibrium computation and Lipschitz constant estimation of recurrent neural networks, yielding novel iterations and tighter upper bounds on Lipschitz constants via forward-backward splitting.

Topics of future research include (i) extending results to more general Banach spaces with a focus on $L_{1}$ and $L_{\infty}$ spaces, (ii) studying the convergence of additional operator splitting methods such as forward-backward-forward~\citep{PT:00} and Davis-Yin~\citep{DD-WY:17} splittings, (iii) extending the theory to variable step size methods, and (iv) considering additional machine learning applications such as $\ell_{\infty}$ robustness of deep neural networks as a non-Euclidean analog of~\citep{PLC-JCP:20} or reinforcement learning and dynamic programming, where $\ell_{\infty}$ contractive and nonexpansive operators are prevalent; see the recent work~\citep{JL-EKR:23} for preliminary ideas in this direction.

\acks{
This material was supported
in part by the National Science Foundation Graduate Research Fellowship under grant 2139319 and in part by AFOSR
under grant FA9550-22-1-0059.}




\vskip 0.2in
\bibliography{alias,Main,FB,New}

\end{document}